\documentclass[aip,jmp,amsmath,amssymb,reprint]{revtex4-1}

\usepackage{bbm}
\usepackage{xypic}
\usepackage{braket}
\usepackage{textcomp}
\usepackage{amsthm}
\usepackage{amsmath}
\usepackage{amssymb}
\usepackage{wasysym}
\usepackage{enumitem}
\usepackage{nicefrac}
\usepackage{stmaryrd}
\usepackage{mathrsfs}
\usepackage{mathtools}

\newcounter{main}
\newtheorem{prop}[main]{Proposition}
\newtheorem{cor}[main]{Corollary}
\newtheorem{thm}[main]{Theorem}
\newtheorem{lem}[main]{Lemma}

\newtheorem{prob}[main]{Problem}
\theoremstyle{definition}
\newtheorem{dfn}[main]{Definition}
\newtheorem{trm}[main]{Terminology}

\newtheorem{exa}[main]{Example}

\theoremstyle{remark}
\newtheorem{rem}[main]{Remark}

\newcommand{\R}{{\mathbb{R}}}
\newcommand{\N}{{\mathbb{N}}}
\newcommand{\C}{{\mathbb{C}}}

\newcommand{\id}{{\mathrm{id}}}

\newcommand{\asrt}{{\operatorname{asrt}}}

\newcommand{\Cat}[1]{\mathsf{#1}}
\newcommand{\op}{\textsf{op}}

\newcommand{\vN}{\Cat{vN}}

\newcommand{\aandthen}{\mathbin{\tilde{*}}}

\newcommand{\ceil}[1]{{\lceil #1 \rceil}}

\newcommand{\after}{\mathbin{\circ}}
\newcommand{\Eff}[1]{[0,1]_{#1}}

\newcommand{\sa}[1]{{#1}_{\mathrm{sa}}}
\newcommand{\pos}[1]{{#1}_{+}}

\newcommand{\uwlim}{\qopname\relax m{uwlim}}

\DeclarePairedDelimiter\floor{\lfloor}{\rfloor}
\DeclareMathOperator{\tr}{tr}

\begin{document}

\title{A universal property for sequential measurement}

\date{\today}

\author{Abraham Westerbaan}
\email{bram@westerbaan.name}
\author{Bas Westerbaan}
\email{bas@westerbaan.name}
\affiliation{Radboud Universiteit Nijmegen}

\maketitle
We study the \emph{sequential product}
\cite{gudder2,gudder3,gudder4},
the operation $p*q = \sqrt{p}q\sqrt{p}$
on the set of effects, $[0,1]_\mathscr{A}$,
of a von Neumann algebra~$\mathscr{A}$
that represents sequential measurement
of first~$p$ and then~$q$.
In~\cite{gudder}
Gudder and Lat\'emoli\`ere give
a list of axioms
based on physical grounds
that completely determines
the sequential product on
a von Neumann algebra of type~I,
that is, a von Neumann algebra~$\mathscr{B}(\mathscr{H})$
of all bounded operators on some Hilbert space~$\mathscr{H}$.
In this paper
we give a list of axioms
that completely determines
the sequential product
on all von Neumann algebras simultaneously,
see Thm~\ref{thm:main}.

These axioms
may be formulated
in purely categorical terms
(although we do not pursue this here,
see also Remark~\ref{rem:cats}).
In this way
this paper contributes to the larger program~\cite{nd,telos,effectusintro}
to identify structure
in the category of von Neumann algebras
with completely positive normal linear contractions
to interpret
the  constructs
in a programming language
designed for a quantum computer:
with the sequential product
one can interpret
measurement.\cite{telos,effectusintro}

Our axioms for the sequential product
are based on the following observations.
Given a von Neumann algebra~$\mathscr{A}$
and~$p\in [0,1]_\mathscr{A}$
the expression $\sqrt{p}a\sqrt{p}$
makes sense for all~$a\in \mathscr{A}$
(and not only for~$a\in [0,1]_\mathscr{A}$).
The resulting map~$\asrt_p\colon \mathscr{A}\to\mathscr{A}$
(so $\asrt_p(a)=\sqrt{p}a\sqrt{p}$)
factors as
\begin{equation*}
\xymatrix@C=7em{
\mathscr{A} 
\ar[r]^-{\pi\colon a\mapsto \ceil{p}a \ceil{p}}
&
\ceil{p}\mathscr{A} \ceil{p}
\ar[r]^-{c\colon a\mapsto \sqrt{p}a\sqrt{p}}
&	
\mathscr{A}
},
\end{equation*}
where $\ceil{p}$ is the least projection above~$p$.

(Roughly speaking,
the von Neumann algebra $\ceil{p}\mathscr{A}\ceil{p}$
represents the subtype of~$\mathscr{A}$ in which the predicate~$p$ holds.
The map~$c$ is simply the restriction of~$\asrt_p$ 
to~$\ceil{p}\mathscr{A}\ceil{p}$,
while~$\pi$
is the map which forgets that~$p$ holds.
The map~$c$
is a more sharply typed version of sequential product than~$\asrt_p$
--- 
much in the same way that the absolute value
on the reals is more sharply described
as a map~$\R\to [0,\infty)$ 
than as a map~$\R\to\R$.)

The maps~$c$ and~$\pi$ have a universal property:
$c$ is a \emph{compression of~$p$}
and~$\pi$ is a \emph{corner of~$\ceil{p}$}
(see Definition~\ref{dfn:main}).
Our first axiom for the sequential product~$(p,q)\mapsto p*q$
will be  that  $p*(-) = \tilde \pi \circ \tilde c$
where~$\tilde \pi$ is a corner of~$\ceil{p}$
and~$\tilde c$ is a compression of~$p$.
Somewhat to our surprise,
while~$\tilde\pi$ 
	and~$\tilde c$ are unique up to unique isomorphism,
the composition~$\tilde \pi \circ \tilde c$ 
is not uniquely determined.
To mend this problem,
we add three more axioms.

\begin{trm}
Although we assume the reader is familiar with the basics
of the theory of von Neumann algebras,~\cite{sakai1971}
we have included the
relevant definitions
and a selection of useful results
in the appendix.

For brevity,
a linear map between von Neumann algebras,
which is  normal, completely positive,
and contractive,
will be called a \textbf{process}.
(This generalizes the standard notion of
\emph{quantum process} between finite-dimensional Hilbert
spaces to von Neumann algebras.)
\end{trm}

\begin{dfn}
\label{dfn:main}
Let~$\mathscr{A}$
and~$\mathscr{C}$ 
be a von Neumann algebras,
and let~$p\in \mathscr{A}$
with $0\leq p \leq 1$ be given.
\begin{enumerate}
\item
A map $\tilde c\colon \mathscr{C} \to\mathscr{A}$
is a \textbf{compression} of~$p$
if~$\tilde c$ is a process
with
$\tilde c(1)\leq p$, and~$\tilde c$
is final among such maps
in the sense that
for every von Neumann algebra~$\mathscr{B}$
and process~$f\colon \mathscr{B} \to \mathscr{A}$
    with $f(1) \leq p$
there is a unique process
$\bar{f} \colon \mathscr{B}
\to \mathscr{C}$
such that~$\tilde c \after \bar{f} = f$. 
\item
A map $\tilde \pi\colon \mathscr{A}\to 
\mathscr{C}$
is a \textbf{corner} of~$p$
if~$\tilde \pi$ is a process
with
$\tilde \pi(p)=\tilde \pi(1)$,
and~$\tilde \pi$ is
initial among such maps
in the sense that
for every von Neumann algebra~$\mathscr{B}$
and process $g\colon \mathscr{A} \to \mathscr{B}$
with~$g(p) = g(1)$
there is a unique 
process~$\bar{g} \colon 
\mathscr{C} \to \mathscr{B}$
with $\bar{g} \after \tilde \pi = g$.
\end{enumerate}
\end{dfn}
\begin{dfn}
\label{dfn:asp}
An \textbf{abstract sequential product}
is a 
family of operations
$\aandthen\colon [0,1]_\mathscr{A}\times
[0,1]_\mathscr{A}\to[0,1]_\mathscr{A}$,
where~$\mathscr{A}$ ranges over all von Neumann algebras,
which obeys the following axioms.
\begin{enumerate}
\item[Ax.1]
For every von Neumann algebra~$\mathscr{A}$
and~$p\in [0,1]_\mathscr{A}$,
there
is a compression~$\tilde c\colon \mathscr{C}\to \mathscr{A}$
of~$p$,
and corner $\tilde \pi\colon \mathscr{A}\to\mathscr{C}$
of~$\ceil{p}$
such that
for all~$q\in[0,1]_\mathscr{A}$,
\begin{equation*}
p \aandthen q \ =\  \tilde c(\tilde \pi(q)).
\end{equation*}

\item[Ax.2]
$p\aandthen(p\aandthen q)=(p\aandthen p)\aandthen q$
for every von Neumann algebra~$\mathscr{A}$
	and all $p,q\in[0,1]_\mathscr{A}$.

\item[Ax.3]
$f(p\aandthen q)=f(p)\aandthen f(q)$
for every multiplicative process
$f\colon \mathscr{A}\to\mathscr{B}$
and all $p,q\in [0,1]_\mathscr{A}$.

\item[Ax.4]
For every von Neumann algebra~$\mathscr{A}$
and
$p\in [0,1]_\mathscr{A}$,
and projections $e_1,e_2\in\mathscr{A}$,
\begin{equation*}
p \aandthen e_1 \ \leq\ 1-e_2
\quad\iff\quad p\aandthen e_2 \ \leq\ 1-e_1.
\end{equation*}
\end{enumerate}
\end{dfn}

Let us formulate the main result of this paper.
\begin{thm}\label{thm:main}
The
sequential product, $*$, given by
\begin{equation*}
p*q \ = \ \sqrt{p}q\sqrt{p}
\end{equation*}
for every von Neumann algebra~$\mathscr{A}$
	and $p,q\in[0,1]_\mathscr{A}$,
is the unique abstract sequential product
(see Definition~\ref{dfn:asp}).
\end{thm}
The proof of Theorem~\ref{thm:main}
spans the length of this paper.

\section{Corners}
\begin{prop}
	\label{prop:corner}
        \label{prop:fsupp-C}
	Let~$\mathscr{A}$
	be a von Neumann algebra,
	and let~$p\in [0,1]_\mathscr{A}$.
	Then $\pi \colon \mathscr{A}\to\floor{p}\mathscr{A}\floor{p},\ 
	a\mapsto \floor{p}a\floor{p}$
	is a corner of~$p$.
\end{prop}
\begin{proof}
Note that $\floor{p}\mathscr{A}\floor{p}$
is a von Neumann subalgebra of~$\mathscr{A}$
(with unit~$\floor{p}$)
by Corollary~\ref{cor:corner-von-neumann}.
Let us show that~$\pi$ is a process.
To begin, $\pi$
is normal and completely positive,
because the map $\mathscr{A}\to\mathscr{A},\ a\mapsto \floor{p}a\floor{p}$
is normal and completely positive 
by Lemma~\ref{lem:comprcpn}.
Further,
since~$\|\floor{p}\|\leq 1$,
we have  $\|\pi(a)\| = 
\|\floor{p}a\floor{p}\| \leq \|\floor{p}\|^2 \|a\| \leq \|a\|$,
for all~$a\in\mathscr{A}$,
and so~$\pi$ is contractive.
Hence~$\pi$ is a process.
Further, $\pi(1) = \floor{p} =\floor{p}p\floor{p} = \pi(p)$
by Proposition~\ref{prop:support-neumann}.

To prove that~$\pi$ is a corner of~$p$
it remains to be shown
that~$\pi$ is initial in the sense
that for every process~$g\colon \mathscr{A}\to\mathscr{B}$
with~$g(p)=g(1)$
there is a unique process
$\bar{g}\colon\floor{p}\mathscr{A}\floor{p}\to\mathscr{B}$
with~$\bar{g}\circ \pi = g$.

        \emph{(Uniqueness)}\ 
        Let $\bar g_1,\bar g_2\colon 
	\floor{p}\mathscr{A}\floor{p} \to\mathscr{B}$
        be processes with~$\bar{g}_1\circ \pi = g = \bar{g}_2 \circ \pi $.
	We must show that~$\bar{g}_1=\bar{g}_2$.

        Let~$a\in \floor{p}\mathscr{A}\floor{p}$
        be given.
	Then~$a=\floor{p}a\floor{p}=\pi(a)$,
       	and so~$\bar{g}_1(a)=\bar{g}_1(\pi(a))=g(a)$.
        Similarly $\bar{g}_2(a)=g(a)$,
	and so~$\bar{g}_1(a)=\bar{g}_2(a)$.
	Hence~$\bar{g}_1 = \bar{g}_2$.

        \emph{(Existence)}\ 
        To begin,
	we will prove that~$g(1-\floor{p})=0$.
        Since~$1-\floor{p} = \ceil{1- p}$
        is the supremum
        of~$1-p \leq (1-p)^{\nicefrac{1}{2}} \leq
        (1-p)^{\nicefrac{1}{4}}\leq \dotsb$
	(see Proposition~\ref{prop:support-neumann})
	and~$g$ is normal,
	it suffices to show that
        \begin{equation}
                \label{eq:ceil-approx}
                g( \ (1-p)^{\nicefrac{1}{2^n}}\ )\ = \ 0
                \qquad\text{for all }\ n\in \N.
        \end{equation}
        Note that~$g(1-p)=0$,
        so to prove~\eqref{eq:ceil-approx}
        it suffices to show that~$g(a)=0$ entails~$g(a^{\nicefrac{1}{2}})=0$
        for all~$a\in \mathscr{A}$ with $a\geq 0$.
        Since~$g$ is 2-positive,
        we have (by Theorem~\ref{thm:cs}),
        for all~$b,c\in \mathscr{A}$,
        \begin{equation}
                \label{eq:cauchy-schwarz-f}
                \|g(b^*c)\|^2\ \leq\ 
                \|g(b^*b)\|\, \|g(c^*c)\|.
        \end{equation}
        In particular, for~$a\in\mathscr{A}_+$,
	we have
        \begin{equation*}
                \|g(a^{\nicefrac{1}{2}})\|^2 \ \leq\ 
                \|g(1)\|\, \|g(a)\|.
        \end{equation*}
        So~$g(a)=0$ entails~$\|g(a^{\nicefrac{1}{2}})\|^2=0$,
        and $g(a^{\nicefrac{1}{2}})=0$.
        Thus~$g(1-\floor{p})=0$.

        Recall that~$\floor{p}\mathscr{A}\floor{p}$ is a 
	von Neumann subalgebra of~$\mathscr{A}$.
        Let~$j\colon \floor{p}\mathscr{A} \floor{p} \to \mathscr{A}$
        be the inclusion.
	Then~$j$ is a normal contractive $*$-homomorphism,
	and thus a process.

        Define~$\bar{g}:=g\circ j\colon 
	\floor{p}\mathscr{A}\floor{p}\to \mathscr{B}$.
	Then~$\bar{g}$ is a process.
	To complete the proof,
	we must show that~$\bar{g}\circ \pi = g$,
        that is, $g(\floor{p}a\floor{p})=g(a)$ for all~$a\in\mathscr{A}$.

        Let~$a\in \mathscr{A}$ be given.
        We show that $g(\floor{p}a)=g(a)$.
        By the Cauchy--Schwarz inequality for $2$-positive maps
        (see Statement~\eqref{eq:cauchy-schwarz-f}),
        we have,
        \begin{equation*}
                \|g((1-\floor{p}) a)\|^2\  \leq \ 
        \|g(1-\floor{p})\| \,\|g(a^*a)\|.
        \end{equation*}
        Since~$g(1-\floor{p})=0$,
	we have~$\|g((1-\floor{p}) a)\|^2 \leq 0$,
	and so $0= g((1-\floor{p}) a) = g(a)-g(\floor{p}a)$.
        Thus $g(a)= g(\floor{p}a)$.

        Similarly, $g(a)=g(a\floor{p})$ 
	and so $g(\pi(a))=g(\floor{p}a\floor{p}) = g(\floor{p}a) = g(a)$
	for all~$a\in\mathscr{A}$.

	Hence~$\pi$ is a corner for~$p$.
\end{proof}

\section{Compressions}
\begin{prop}
	\label{prop:compression}
        \label{prop:initial-support}
	Let~$\mathscr{A}$
	be a von Neumann algebra,
	and let~$p\in [0,1]_\mathscr{A}$.
	Then $c\colon \ceil{p}\mathscr{A}\ceil{p}
	\to\mathscr{A},\ a\mapsto \sqrt{p}a\sqrt{p}$
	is a compression of~$p$.
\end{prop}
\begin{proof}
Note that $\ceil{p}\mathscr{A}\ceil{p}$
is a von Neumann subalgebra of~$\mathscr{A}$
with unit~$\ceil{p}$
(see Corollary~\ref{cor:corner-von-neumann}).
Since therefore the inclusion~$\ceil{p}\mathscr{A}\ceil{p}\to\mathscr{A}$
is a process,
and the map
$a\mapsto \sqrt{p}a\sqrt{p}\colon 
\mathscr{A}\mapsto\mathscr{A}$
is a process (see Lemma~\ref{lem:comprcpn}),
it follows that~$c$ is a process.
Further,
note that~$c(1)= p \leq p$.

To prove that~$c$ is a compression
it remains to be
shown that~$c$ is final
in the sense that for 
every von Neumann algebra~$\mathscr{B}$
and process
$f\colon \mathscr{B}\to\mathscr{A}$
with $f(1)\leq p$
there is a unique~$\bar{f}\colon \mathscr{B}\to
\ceil{p}\mathscr{A}\ceil{p}$
such that~$f=c\circ \bar{f}$.

\emph{(Existence)}\ 
Note that if~$\sqrt{p}$
is invertible in~$\mathscr{A}$,
then
we can define $\bar{f}\colon \mathscr{B}\to
\ceil{p}\mathscr{A}\ceil{p}$
by, for all~$b\in\mathscr{B}$,
\begin{equation*}
\bar{f}(b)\ =\ \sqrt{p}^{-1} f(b)\sqrt{p}^{-1},
\end{equation*}
and this does the job.
Also,
if~$\sqrt{p}$ is pseudoinvertible
---  $q\sqrt{p} = \sqrt{p}q = \ceil{p}$
for some~$q\in\mathscr{A}$ ---,
then~$\bar{f}$ can be defined in a similar manner.
However,
$p$ might not be pseudoinvertible.\footnote{The operator~$T$
on the Hilbert space~$\ell^2$ given by $T(s)(n)=2^{-n}s(n)$
for $s\in \ell^2$ and~$n\in \N$ is not pseudoinvertible
in~$\mathscr{B}(\ell^2)$.}
Therefore,
we will instead approximate the (possibly non-existent) pseudoinverse
of~$\sqrt{p}$
by a sequence~$q_1,q_2,\dotsc$ in~$\mathscr{A}$
--- much in the same way that an approximate identity
in a $C^*$-algebra
approximates a (possibly non-existent) unit ---,
and define, for~$b\in\mathscr{B}$,
\begin{equation}
\label{eq:def-f-bar}
\bar{f}(b) \ =\  \uwlim_{n\to \infty} q_n f(b)q_n.
\end{equation}
By the Spectral Theorem (see Thm.~\ref{thm:spectral}
and Thm.~\ref{thm:KGNS}),
we may assume without loss of generality
that~$\mathscr{A}$
is a von Neumann subalgebra of
the bounded operators~$\mathscr{B}(L^2(X))$
on the Hilbert space~$L^2(X)$
of square-integrable functions\footnote{Of course,
the elements of~$L^2(X)$
which are equal almost everywhere are identified.}
on some measure space~$X$,
and that
there is a real bounded integrable function~$\hat{p}$ on~$X$
such that, for all $f\in L^2(X)$,
\begin{equation*}
p(f) \ =\  \int \hat{p}\cdot f\,d\mu.
\end{equation*}

Let~$\varrho\colon L^\infty(X)\rightarrow \mathscr{A}$
be given by~$\varrho(g)(f)=\int f\cdot g\,d\mu$
for all~$g\in L^\infty(X)$ and~$f\in  L^2(X)$,
where $L^\infty(X)$
is the von Neumann algebra
of bounded measurable functions\footnote{Of course,
the elements of~$L^\infty(X)$
which are equal almost everywhere are identified.} on~$X$.
Then~$\varrho$ is an injective normal $*$-homomorphism,
and~$\varrho(\hat{p})=p$.

Note that $\sqrt{\hat{p}}$ might not be pseudoinvertible
in~$L^\infty(X)$,
because the function~$\hat{q}: X\to \mathbb{R}$ given by
for~$x\in X$,
\begin{equation*}
\hat{q}(x) \ = \ 
\begin{cases}
p(x)^{-\nicefrac{1}{2}} & \text{if }p(x)\neq 0 \\
0 & \text{if }p(x)=0.
\end{cases}
\end{equation*}
might not be (essentially) bounded.
Nevertheless,
$\sqrt{\hat{p}}\cdot \mathbf{1}_{Q_n}$
has~$\hat{q}\cdot \mathbf{1}_{Q_n}$
as pseudoinverse in~$L^\infty(X)$,
where
\begin{equation*}
Q_n \ =\  \{\,x\in X\colon\, \hat{p}(x)>\nicefrac{1}{n}\,\}
\ =\ \hat{p}^{-1}(\,(\nicefrac{1}{n},1]\,).
\end{equation*}
Define~$q_n = \varrho(\hat{q}\cdot \mathbf{1}_{Q_n})$
for all~$n\in \N$.

Let~$b\in \mathscr{B}$ be given.
We want to define~$\bar{f}(b)$ by Equation~\eqref{eq:def-f-bar},
but for this, we must first show that
$(q_n \,f(b)\,q_n)_n$ converges ultraweakly.
It suffices
to show that $(q_n\,f(b)\,q_n)_n$
is norm bounded,
and  ultraweakly Cauchy
(see Proposition~\ref{prop:ultraweak-convergence}).

We only need to consider the case that~$b\in [0,1]_\mathscr{B}$.
Indeed, 
any $b\in\mathscr{B}$ can written as
\begin{equation}
\label{eq:reduce-b}
b\,\equiv\,\|b\|\,(b_1-b_2+ib_3-ib_4),
\end{equation}
where~$b_i\in[0,1]_\mathscr{B}$,
and if $(q_n \,f(b_i) \,q_n)_n$
converges ultraweakly for each~$i$, then
so does $(q_n\,f(b)\,q_n)_n$.

Let~$n\in\mathbb{N}$ be given.
Since~$f(b)\leq f(1)\leq p$,
we have~$q_n \,f(b)\,q_n \leq q_n \,p \,q_n$.
Since $q_n = \varrho(\hat{q}\cdot \mathbf{1}_{Q_n})$,
$p = \varrho(\hat{p})$,
and~$\hat{q}\cdot \mathbf{1}_{Q_n}$
is the pseudoinverse of~$\sqrt{\hat{p}}\cdot\mathbf{1}_{Q_n}$,
we get~$q_n\,p\,q_n = \varrho(\mathbf{1}_{Q_n})\leq 1$,
and so~$q_n\,f(b)\,q_n \leq 1$.
Hence~$\|q_n\,f(b)\,q_n\|\leq 1$,
and so~$(q_n\,f(b)\,q_n)_n$ is norm bounded.

Let~$\varphi\colon \mathscr{A}\to\mathbb{C}$
be a normal state.
To prove that~$(q_n\,f(b)\,q_n)_n$ is ultraweakly Cauchy,
we must show that~$(\varphi(q_n\,f(b)\,q_n))_n$ is Cauchy.

For brevity, define for~$n>m>0$,
\begin{align*}
S_{n,m}\  &=\  
\hat{p}^{-1}(\, (\nicefrac{1}{n},\nicefrac{1}{m}] \,) \\
S_{\infty,m} \ &=\  \hat{p}^{-1}(\, (0, \nicefrac{1}{m}]\,) \\
s_{n,m} \ &=\  \varrho(\,\hat{q}\cdot \mathbf{1}_{S_{n,m}} \,)
\end{align*}
Note that~$S_{n,1}=Q_n$ and~$s_{n,1}=q_n$.
(We have not defined~$s_{\infty,m}
= \varrho(\hat{q}\cdot \mathbf{1}_{S_{\infty,m}})$,
because $\hat{q}\cdot \mathbf{1}_{S_{\infty,m}}$
might not be bounded.)
Note that
\begin{equation}
\label{eq:sps}
s_{n,m} \sqrt{p}  \ =\ 
\sqrt{p} s_{n,m} \ = \ 
\varrho(\mathbf{1}_{S_{n,m}}).
\end{equation}

Let~$0 < m < n$ be given. 
Since~$(\nicefrac{1}{n},1]$
is the disjoint union of~$(\nicefrac{1}{n},\nicefrac{1}{m}]$
and~$(\nicefrac{1}{m},1]$,
$Q_n$ is the disjoint union
of~$S_{n,m}$ and~$Q_m$,
and~$q_n = s_{n,m}+q_m$, and
\begin{alignat*}{3}
&q_n\,f(b)\,q_n \,-\, q_m\,f(b)\,q_m\\
&\qquad = \ s_{n,m}\,f(b)\,s_{n,m}
\,+\, s_{n,m}\,f(b)\,q_m
\,+\, q_m\,f(b)\,s_{n,m}.
\end{alignat*}
Thus,
\begin{equation}
\label{eq:tobound}
\begin{alignedat}{3}
&|\,\varphi(\ q_n\,f(b)\,q_n - q_m\,f(b)\,q_m\ )\,|\\
&\qquad\leq \ |\varphi(s_{n,m}\,f(b)\,s_{n,m})|
\,+\, |\varphi(s_{n,m}\,f(b)\,q_m)|\\
&\qquad\qquad\qquad\,+\, |\varphi(q_m\,f(b)\,s_{n,m})|
\end{alignedat}
\end{equation}
Note that
for 
$k < \ell$ and 
$m < n$,
we have
\begin{alignat*}{3}
    & \lvert \varphi( s_{\ell,k} f(b) s_{n,m} )\rvert^2  \\
        \ &=\  \bigl|\varphi\bigl(\ 
        (\sqrt{f(b)} s_{\ell,k})^* \ \sqrt{f(b)} s_{n,m}\ \bigr)\bigr|^2 \\
        \ &\leq\  \varphi(s_{\ell,k} \,f(b)\, s_{\ell,k}) \ 
        \varphi(s_{n,m}\,f(b)\,s_{n,m}) 
        &\ &\text{by Ineq.~\eqref{eq:kadineq}}\\
        \ &\leq\  \varphi(s_{\ell,k} \,p\, s_{\ell,k}) \ 
        \varphi(s_{n,m}\,p\,s_{n,m}) 
        &&\text{since~$f(b) \leq p$}\\
        \ &=\  \varphi(\varrho(\mathbf{1}_{S_{\ell,k}})) \ 
        \varphi(\varrho(\mathbf{1}_{S_{n,m}})) 
            &&\text{by Eq.~\eqref{eq:sps}}\\
        \ &\leq\  \varphi(\varrho(\mathbf{1}_{S_{n,m}})) 
        &&\text{as~$\mathbf{1}_{S_{\ell,k}}  \leq \mathbf{1}$} \\
        \ &\leq\  \varphi( \varrho(\mathbf{1}_{S_{\infty,m}}) )
        &&\text{as~$S_{n,m} \subseteq S_{\infty,m}$}
\end{alignat*}
Thus
using Eq.~\eqref{eq:tobound}
and~$q_n=s_{n,1}$
we get the bound
\begin{equation}
    \lvert \varphi (\ q_n\, f(b) \,q_n
                    \,- \, q_m \,f(b)\, q_m \ )\rvert 
                    \ \leq \ 
		    3 \sqrt{\varphi(\varrho(\mathbf{1}_{S_{\infty,m}}))}.
                \label{eq:tobound2}
\end{equation}
Since~$(0,1]\supseteq (0,\nicefrac{1}{2}]
\supseteq (0,\nicefrac{1}{3}]\supseteq \dotsb$
and~$\bigcap_m(0,\nicefrac{1}{m}]=\varnothing$,
we have~$S_{\infty,1} \supseteq S_{\infty,2} \supseteq \ldots$
and~$\bigcap_mS_{\infty,m} = \varnothing$.
Then $\inf_m \mathbf{1}_{S_{\infty,m}}=0$,
and so $\inf_m \varphi(\varrho (\mathbf{1}_{S_{\infty,m}}))=0$,
because~$\varrho$ and~$\varphi$ are normal.
Thus $(\, \varphi(\varrho (\mathbf{1}_{S_{\infty,m}}))\,)_m$
converges to~$0$,
and so $(\, \sqrt{\varphi(\varrho (\mathbf{1}_{S_{\infty,m}}))}\,)_m$
converges to~$0$ as well.

Let~$\varepsilon >0$ be given.
There is~$N\in \N$
such that for all~$n > N$, we
have~$\smash{\sqrt{\varphi(\varrho(\mathbf{1}_{S_{\infty,m}}))}} 
\leq \nicefrac{\varepsilon}{3}$.
Then given~$n > m > N$, we have,  by Equation~\eqref{eq:tobound2},
\begin{equation}
\label{eq:bound-qfbq}
\lvert \varphi(\ q_n\, f(b)\, q_n
            \ -\  q_m \,f(b)\, q_m\ )\rvert \ \leq\  \varepsilon.
\end{equation}
Hence~$(q_n f(b)q_n)_n$
is ultraweakly Cauchy and norm bounded,
and must therefore converge ultraweakly.
We may now (and do) define~$\overline{f}(b)$ 
as in Equation~\eqref{eq:def-f-bar}.

Thus, $(q_n f(-)q_n)_n$ 
converges coordinatewise ultraweakly to~$\overline{f}$.
Note that the number~$N$ related to 
Inequality~\eqref{eq:bound-qfbq}
depends on~$\varepsilon$ and~$\varphi$,
but does not depend on~$b$.
It follows that
on~$[0,1]_\mathscr{B}$
the sequence $(q_n f(-)q_n)_n$
converges uniformly ultraweakly to~$\overline{f}$.

It is easy to see that~$\overline{f}$
is linear and positive.
It remains
to be shown that~$\overline{f}$
is contractive, normal, completely positive,
$c\circ \overline{f} = f$,
and $\overline{f}(\mathscr{B})\subseteq \ceil{p}\mathscr{A}\ceil{p}$.

\emph{($\overline{f}(\mathscr{B})\subseteq\ceil{p}\mathscr{A}\ceil{p}$)}\ 
Let~$b\in\mathscr{B}$ be given.
We must show that~$\overline{f}(b)\in\ceil{p}\mathscr{A}\ceil{p}$.
By writing~$b$ as in Equation~\eqref{eq:reduce-b},
the problem is easily reduced to the case that~$b\in [0,1]_\mathscr{B}$.

Let~$n\in \N$ be given.
Since~$b\leq 1$, we have~$f(b)\leq f(1)\leq p$, 
and so~$q_n f(b) q_n \leq q_n p q_n = \varrho(\mathbf{1}_{Q_n})
= \varrho(\mathbf{1}_{S_{\infty,1}})$.
Since~$\hat{p}^{-1}((0,1]) = S_{\infty,1}$,
and it is not hard to see that~$\mathbf{1}_{\hat{p}^{-1}((0,1])}$
is the support of~$\hat{p}$ in~$L^\infty(X)$,
it follows that~$\varrho(\mathbf{1}_{S_{\infty,1}})=\ceil{p}$
(see Proposition~\ref{prop:awmult}).
Thus~$q_n f(b) q_n \leq \ceil{p}$ for all~$n$,
and so~$\overline{f}(b) = \uwlim_n q_n f(b) q_n  \leq \ceil{p}$.
Corollary~\ref{cor:proj},
gives us~$\ceil{p}\overline{f}(b)\ceil{p} = \overline{f}(b)$,
and so~$\overline{f}(b)\in \ceil{p}\mathscr{A}\ceil{p}$.

\emph{($\overline{f}$ is contractive)}\ 
It suffices to show that~$\overline{f}(1)\leq 1$.
Let~$n\in \N$ be given.
Since $f(1)\leq p$,
we have 
\begin{equation*}
q_n f(1)q_n \ \leq \ q_npq_n \ =\  \varrho(\mathbf{1}_{Q_n}) \ \leq\ 1.
\end{equation*}
Thus~$\overline{f}(1) = \uwlim_{n} q_n f(1) q_n \leq 1$.

 \emph{($c\circ \overline{f}=f$)}\ 
Let~$b\in[0,1]_\mathscr{B}$.
It suffices to show that~$c(\overline{f}(b))=f(b)$.
 Since~$c$ is normal, we have
\begin{equation*}
c(\overline{f}(b))
\ = \ \uwlim_{n\to \infty}\sqrt{p}q_n\, f(b)\, q_n\sqrt{p}.
\end{equation*}
Thus we must show that $(\,\sqrt{p}q_n\, f(b)\,q_n\sqrt{p}\,)_n$
converges ultraweakly to~$f(b)$.

Let~$n\in \N$ be given.
On the one hand we have~$\sqrt{p}q_n = \varrho(\mathbf{1}_{Q_n})
= \varrho(\mathbf{1}_{S_{n,1}})$
by definition of~$q_n$.
On the other hand
we have
$\ceil{p} f(b)\ceil{p} = f(b)$
and~$\ceil{p} 
= \varrho( \mathbf{1}_{S_{\infty,1}})$.
Thus,
using~$\mathbf{1}_{S_{\infty,1}}
= \mathbf{1}_{S_{\infty,n}}+\mathbf{1}_{S_{n,1}}$,
and writing $e_{k,\ell} = \varrho(\mathbf{1}_{S_{k,\ell}})$,
we have
\begin{equation}
\label{eq:pqnfbqnp}
\begin{alignedat}{3}
&f(b) \,-\, \sqrt{p}q_n \,f(b)\,q_n\sqrt{p} \\
\ &\qquad\qquad=\ 
e_{\infty,n}	\,f(b)\,  e_{\infty,n}
\,+\, e_{n,1} \,f(b)\, e_{\infty,n} \\
 &\qquad\qquad\qquad\qquad \,+\, e_{\infty,n} \,f(b)\, 
	e_{n,1}
\end{alignedat}
\end{equation}
So to show that $(\,\sqrt{p}q_n\, f(b)\,q_n\sqrt{p}\,)_n$
converges ultraweakly to~$f(b)$,
it suffices to show that the terms on the right-hand side
of Equation~\eqref{eq:pqnfbqnp} converge ultraweakly to~$0$.
Let~$\varphi\colon \mathscr{A}\to\C$ be a normal state.
Then
\begin{alignat*}{3}
&|\varphi(\, 
e_{n,1}
\,f(b)\, 
e_{\infty,n}
\,)|^2\\
\ &=\ 
|\varphi(\, 
(\sqrt{f(b)} e_{n,1})^*\, 
\sqrt{f(b)} e_{\infty,n}\,)
|^2 \\
\ &\leq\ 
\varphi(e_{n,1} f(b) e_{n,1})
\,\cdot\, \varphi(e_{\infty,n} f(b) e_{\infty,n})
\qquad&&\text{by Ineq.~\eqref{eq:kadineq}}\\
\ &\leq\ 
\varphi(e_{n,1})\,\cdot\,
\varphi(e_{\infty,n})
\qquad&&\text{since $f(b)\leq 1$}\\
\ &\leq\ 
\varphi(e_{\infty,n})
\qquad&&\text{since $e_{n,1}\leq 1$.}
\end{alignat*}
Recall that $\varphi(e_{n,\infty})
\equiv \varphi(\varrho(\mathbf{1}_{S_{\infty,n}}))$
converges to zero (because $\bigcap_n S_{\infty,n}=\varnothing$).
It follows that $(\,e_{n,1}f(b)e_{\infty,n}\,)_n$
converges ultraweakly to~$0$.

By a similar reasoning,
$(\,e_{\infty,n}f(b)e_{n,1}\,)_n$
and
$(\,e_{\infty,n}f(b)e_{\infty,n}\,)_n$
converge ultraweakly to~$0$.
Thus, by Equation~\eqref{eq:pqnfbqnp}, 
$(\,\sqrt{p}q_n\, f(b)\,q_n\sqrt{p}\,)_n$
converges ultraweakly to~$f(b)$.
Thus~$c\circ \overline{f} = f$.

\emph{($f$ is normal)}\ 
Since
$(q_n f(-)q_n)_n$
converges uniformly ultraweakly on~$\Eff{\mathscr{B}}$
to~$\overline{f}$,
and each~$q_n f(-) q_n$ is normal
(by Lemma~\ref{lem:comprcpn}),
it follows that
$\overline{f}$ is normal
(by Corollary~\ref{cor:uw-lim-normal}).

\emph{($f$ is completely positive)}\ 
Since
$(q_n f(-)q_n)_n$
converges coordinatewise ultraweakly
to~$\overline{f}$,
and each~$q_n f(-)q_n$
is completely positive 
(see Lemma~\ref{lem:comprcpn}),
it follows that
$\overline{f}$ is completely positive (by Corollary~\ref{cor:uw-lim-cp}).

\emph{(Uniqueness)}\ 
Let~$g\colon \mathscr{B}\to\ceil{p}\mathscr{A}\ceil{p}$
be a process with~$c\circ g = f$.
We must show that~$g=\overline{f}$.

Let~$b\in [0,1]_\mathscr{B}$ be given.
It suffices to show that~$\overline{f}(b)=g(b)$.
We have $\sqrt{p}g(b) \sqrt{p} = f(b)$.
Let~$n\in \N$ be given.
We have $e_{n,1} g(b) e_{n,1} = q_n\sqrt{p} g(b) \sqrt{p}q_n
= q_n f(b) q_n$
since~$q_n\sqrt{p}= \varrho(1_{S_{n,1}})\equiv e_{n,1}$.
On the one hand $(q_n f(b) q_n)_n$
converges ultraweakly to~$\overline{f}(b)$
by definition of~$\overline{f}(b)$.
On the other hand
$(e_{n,1} g(b) e_{n,1})_n$ converges ultraweakly
to~$g(b)$
as one can see with tricks that were used before.
Hence~$\overline{f}(b)=g(b)$.
\end{proof}

\section{Existence}
To show that the sequential product 
is an abstract sequential product,
we use the following result,
which (we think) is interesting
in itself.
\begin{lem}
\label{lem:connected}
Let~$a$ be an element of a von Neumann algebra
(or a unital $C^*$-algebra)~$\mathscr{A}$
with~$a^*a \leq 1$.
Then for projections $e_1,e_2\in\mathscr{A}$
the following are equivalent.
\begin{enumerate}
\item
\label{lem:connected-1}
$a^*e_1a \leq 1- e_2$
\item
\label{lem:connected-2}
$ae_2a^* \leq 1- e_1$
\item
\label{lem:connected-3}
$e_1ae_2=0$
\item
\label{lem:connected-4}
$e_2a^*e_1=0$
\end{enumerate}
\end{lem}
\begin{proof}
\emph{(1 $\Longrightarrow$ 3)}\ 
We must show that~$e_1ae_2=0$.
It suffices to show $e_2a^*e_1ae_2=0$,
because $\|e_1ae_2\|^2  = \|e_2 a^* e_1 a e_2\|$
by the C$^*$-identity.
Since~$0\leq a^*e_1 a\leq 1- e_2$,
we have $0\leq e_2 a^* e_1 a e_2 \leq  e_2 (1-e_2 )e_2 = 0$,
and so $e_2 a^*e_1ae_2=0$.

\emph{(3 $\Longrightarrow$ 1)}\ 
Since $e_1ae_2=0$, also $e_1 a = e_1 a (1-e_2)$,
and  $a^* e_1 = (1-e_2) a^* e_1$.
Then $a^*e_1 a = (1-e_2) a^* e_1 a (1-e_2) \leq 1-e_2$,
because $a^* e_1 a \leq a^*a \leq 1$.

\emph{(4 $\Longleftrightarrow$ 2)}\ 
follows by the same reasoning as~1 $\Longleftrightarrow$ 3.

\emph{(3 $\Longleftrightarrow$ 4)}\ 
follows by applying~$(-)^*$.
\end{proof}
\begin{prop}\label{prop:existence}
The sequential product~$*$
(which is given by  $p*q= \sqrt{p}q\sqrt{p}$)
is an abstract sequential product. 
\end{prop}
\begin{proof}
\emph{(Ax.1)}\ 
Let~$\mathscr{A}$
be a von Neumann algebra,
and let~$p,q\in [0,1]_\mathscr{A}$.
Since~$\ceil{p}\sqrt{p}=\sqrt{p}$
(by Prop.~\ref{prop:support-neumann}),
\begin{equation*}
p*q\ = \ 
\sqrt{p}q\sqrt{p}
\ =\ \sqrt{p}\ceil{p}q\ceil{p}\sqrt{p}
\ =\ c(\pi{p}(q)), 
\end{equation*}
where~$\pi{p}\colon \mathscr{A}\to \ceil{p}\mathscr{A}\ceil{p}$ 
is the corner of~$\ceil{p}$ from Proposition~\ref{prop:fsupp-C},
and $c\colon \ceil{p}\mathscr{A}\ceil{p}\to\mathscr{A}$
is the compression of~$p$
from Proposition~\ref{prop:initial-support}.
Thus~$*$ obeys~Ax.1.

The proof of
\emph{(Ax.2)} and \emph{(Ax.3)} is easy,
and
\emph{(Ax.4)}\ follows from
Lemma~\ref{lem:connected}.
\end{proof}

\section{Uniqueness}
We will need the following fact later on.
\begin{lem}\label{lem:lalgebra}
   Let~$f,g\colon V \to W$
        be linear maps between complex vector spaces.
    Assume that for every~$v \in V$,
        there is an~$\alpha \in \C \backslash\{0\}$
        with~$f(v) = \alpha\cdot g(v)$.

Then there is~$\alpha_0 \in \mathbb{C}\backslash \{0\}$
with~$f = \alpha_0 \cdot  g$.
\end{lem}
\begin{proof}
For the moment,
assume~$f$ and~$g$ are injective.
If~$V=\{0\}$, then~$\alpha_0 \equiv 1$ works,
so assume~$V \neq \{0\}$.
Pick any~$v \in V$ with~$v \neq 0$.
Let~$\alpha_0\in\mathbb{C}\backslash\{0\}$
 be such that~$f(v) = \alpha_0 \cdot g(v)$.
Let~$w \in V$.
We have to show that $f(w) = \alpha_0 \cdot g(w)$.
Now, either~$g(v)$ and~$g(w)$ are linearly dependent or not.

Suppose that~$g(v)$ and~$g(w)$ are linearly independent.
Let~$\beta \in \C \backslash \{0\}$
be such that~$f(w) = \beta \cdot g(w)$,
and let~$\gamma\in\C\backslash\{0\}$ be such 
that~$f(v+w) = \gamma\cdot g(v+w)$.
Then
\begin{equation*}
    (\gamma - \alpha_0) \cdot g(v) \,+\,  (\gamma - \beta)\cdot g(w)\ = \ 0.
\end{equation*}
By linear independence,
we have~$\gamma - \alpha_0 = 0 = \gamma - \beta$.
Hence~$\alpha_0 = \beta$, and so~$f(w) = \alpha_0 \cdot g(w)$.

Suppose that~$g(v)$ and~$g(w)$ are linearly dependent.
As~$v \neq 0$ and~$g$ is injective,
we have~$g(v) \neq 0$.
Thus~$g(w) = \varrho \cdot g(v)$ for some~$\varrho \in \C$.
Then~$g(w - \varrho \cdot v)=0$,
and so~$w= \varrho\cdot v$,
since~$g$ is injective.
We have
\begin{equation*}
    f(w) \ =\  \varrho\cdot f( v) 
    \ =\  \varrho\cdot\alpha_0 \cdot g(v) 
    \ =\  \alpha_0\cdot g(w).
\end{equation*}

Thus we have $f(w) = \alpha_0 g(w)$
whether $g(v)$ and~$g(w)$ are linearly dependent or not.

We now return to the general case
in which~$f$ and~$g$ might not be injective.
Note that the kernels of~$f$ and~$g$ coincide,
and so, writing~$N \equiv \ker f = \ker g$,
there are unique~$t,s \colon V/N \to W$
such that~$s \after q = f$
and~$t \after q = g$,
where~$q\colon V \to V/N$
is the quotient map.
Clearly, $s$ and~$t$ are injective,
and for every~$v\in V/N$
there is~$\alpha \in \mathbb{C}\backslash \{0\}$
with~$s(v) = \alpha\cdot t(v)$.
Thus, by the previous discussion,
there is $\alpha_0\in\mathbb{C}\backslash\{0\}$
with~$s = \alpha_0 \cdot t$.
Then~$f = \alpha_0\cdot g$.
\end{proof}

\begin{prop}
\label{prop:uniqueness}

For any abstract sequential product, $\aandthen$,
we have $p \aandthen q = \sqrt{p}q\sqrt{p}$,
where~$p,q\in[0,1]_\mathscr{A}$
and~$\mathscr{A}$ is a von Neumann algebra.
\end{prop}
\begin{proof}
Let~$\mathscr{A}$ be a von Neumann algebra,
and~$p \in [0,1]_{\mathscr{A}}$.
By~Ax.1 there is a corner~$\tilde{\pi}$ of~$\ceil{p}$
and a compression~$\tilde{c}$ of~$p$
such that~$p \aandthen q= \tilde{c}(\tilde{\pi}(q))$
for all~$q\in [0,1]_\mathscr{A}$.

Let~$c\colon \ceil{p}\mathscr{A}\ceil{p}\to\mathscr{A}$
be the compression of~$p$ 
given by~$c(a)=\sqrt{p}a\sqrt{p}$ for all~$a\in \ceil{p}\mathscr{A}\ceil{p}$
(see Proposition~\ref{prop:compression}).
Since both~$c$ and~$\tilde c$ are
compressions of~$p$
it is easy to see that there is 
an invertible process $\vartheta$
such that $\tilde c = c \circ \vartheta$.
In fact,
$\vartheta$ is a $*$-isomorphism
by 
Corollary~\ref{cor:invprocmult}.

Similarly,
$\tilde\pi = \chi\circ \pi$
where~$\chi$
is some $*$-isomorphism~$\chi$,
and~$\pi\colon \mathscr{A} \to \ceil{p}\mathscr{A}\ceil{p}$
is the corner of~$\ceil{p}$ 
given by~$\pi(a) = \ceil{p}a\ceil{p}$
for all~$a\in\mathscr{A}$
(see Proposition~\ref{prop:corner}).

Thus~$p\aandthen q = \sqrt{p} \,\psi(\,\ceil{p} q\ceil{p}\,)\, \sqrt{p}$
for all~$q\in[0,1]_\mathscr{A}$,
where~$\psi=\vartheta\circ \chi$
is a $*$-automorphism of~$\ceil{p}  \mathscr{A} \ceil{p}$.

Roughly speaking,
our goal is to prove~$\psi=\id$.
We will first consider the 
case that~$\mathscr{A} = \mathscr{B}(\mathscr{H})$.
Since~$\ceil{p} \mathscr{B}(\mathscr{H}) \ceil{p}$
is a type~I factor (i.e.~$*$-isomorphic
to~$\mathscr{B}(\mathscr{K})$ for some Hilbert
space~$\mathscr{K}$), it is known\cite{kaplansky1952}
that~$\psi$ must be an inner $*$-automorphism, that is,
there is a unitary~$u \in \ceil{p} \mathscr{B}(\mathscr{H}) \ceil{p}$
such that~$\psi(a) = u^* a u$
for all~$a \in \ceil{p} \mathscr{B}(\mathscr{H}) \ceil{p}$.
Note that $\ceil{p}u = u$ since~$u \in \ceil{p}\mathscr{B}(\mathscr{H})\ceil{p}$.
Thus we have, for all~$b \in [0,1]_{\mathscr{B}(\mathscr{H})}$,
\begin{equation}
    p \aandthen b \ =\  \sqrt{p} u^* b u \sqrt{p}.\label{eq:formaandthen}
\end{equation}
We aim to show that~$u=1$,
or at least that $u=\alpha 1$ for some~$\alpha \in \C$
with~$|\alpha|=1$.

Our first step is to prove that~$u p = p u$.
To this end, we extract some information about~$u$
from~Ax.4.
First, note that for vectors~$v,w\in \mathscr{H}$ with~$\|w\|=1$
and~$\|v\| \leq 1$,
\begin{equation}\label{eq:orthvecproj}
\ket{v}\!\bra{v} \ \leq\  1- \ket{w}\!\bra{w}
\quad\text{if and only if}\quad \left<v,w\right>\,=\,0.
\end{equation}
For any $v \in \mathscr{H}$ with~$\|v\|=1$, 
\begin{equation}
\label{eq:aandthenonrankone}
p \aandthen \ket{v}\!\bra{v}
\ =\  \sqrt{p} u^* \ket{v}\!\bra{v} u \sqrt{p}
\ =\  \ket{\sqrt{p} u^* v}\!\bra{\sqrt{p} u^* v}.
\end{equation}
For all~$v,w\in \mathscr{H}$ with  $\|v\|=\|w\|=1$,
the following are equivalent
\begin{alignat*}{3}
\left<\sqrt{p}u^*v,w\right>\ &=\ 0\\
\ket{\sqrt{p}u^*v }\!\bra{\sqrt{p}u^*v} 
\ &\leq\ 1-\ket{w}\!\bra{w}
\qquad &\text{by~\eqref{eq:orthvecproj}} \\
p\aandthen \ket{v}\!\bra{v} 
\ &\leq\ 1-\ket{w}\!\bra{w}
	\qquad&\text{by~\eqref{eq:aandthenonrankone}}\\
p\aandthen \ket{w}\!\bra{w} 
\ &\leq\ 1-\ket{v}\!\bra{v}
\qquad&\text{by~Ax.4}\\
&\vdots\\
\left<\sqrt{p}u^*w,v\right>\ &=\ 0\\
\left<u\sqrt{p}v,w\right>\ &=\ 0
\end{alignat*}
Thus~$\sqrt{p} u^*v$
    and~$u \sqrt{p} v$
    are orthogonal to the same vectors,
and so there is~$\alpha \in \C\backslash\{0\}$
with
\begin{equation*}
    \sqrt{p} u^* v \ =\  \alpha\cdot u \sqrt{p} v.
\end{equation*}
By scaling it is clear that this statement is also true for
all~$v \in \mathscr{H}$ (and not just for~$v$ with~$\|v\|=1$).

Although a priori~$\alpha$ might depend on~$v$,
we know by Lemma~\ref{lem:lalgebra}
that there is an~$\alpha \in \C \backslash \{ 0\}$
such that~$\sqrt{p}u^* = \alpha \cdot u \sqrt{p}$.
It follows that $p = \sqrt{p} u^* u \sqrt{p}
    = \alpha \cdot u \sqrt{p}  u \sqrt{p}
    = u \sqrt{p} \sqrt{p} u^*
    = u p u^*$,
and so~$p u = u p$.
Then also~$\sqrt{p} u = u \sqrt{p}$
(see Corollary~\ref{cor:commutes-sqrt}),
and thus~$\sqrt{p} u^* = \alpha u \sqrt{p} = \alpha \sqrt{p} u$.

Note that~$(\sqrt{p}u^*)^* = u \sqrt{p}$,
and so~$u \sqrt{p} = \alpha^* \sqrt{p} u^* = \alpha^* \alpha u \sqrt{p}$.
Then if~$u\sqrt{p} \neq 0$,
we get~$\alpha^*\alpha = 1$,
and if~$u \sqrt{p}=0$, we can put~$\alpha=1$
and still have both~$\sqrt{p} u^* = \alpha \sqrt{p}u$
and~$\alpha^*\alpha=1$.
It follows that, for all~$b \in \mathscr{B}(\mathscr{H})$,
\begin{equation*}
    c(u^*bu)=
    \sqrt{p} u^* b u \sqrt{p}
        = \sqrt{p} u b u^* \sqrt{p}
        = c(ubu^*),
\end{equation*}
where~$c$ is the compression of~$p$ from Proposition~\ref{prop:compression}.
By the universal property of~$c$ 
we get~$u^* (-) u = u (-)u^*$,
and thus~$u^2b = bu^2$ for all~$b \in \mathscr{B}(\mathscr{H})$.
Hence~$u^2$ is central in~$\mathscr{B(}\mathscr{H})$.
Since~$\mathscr{B}(\mathscr{H})$ is a factor, we get~$u^2 = \lambda \cdot 1$
for some~$\lambda \in \C$ with~$| \lambda | = 1$.

Since~$p$ commutes with~$u$,
we easily get~$p\aandthen p = p^2$.
Then from~Ax.2 it follows that
\begin{align*}
    p^2 \aandthen q\  &=\ (p \aandthen p)*q  \\
    &=\  p * (p \aandthen q) \\ 
    &=\  \sqrt{p} u^* \,\sqrt{p} u^*\,q\,u \sqrt{p}\,
            u \sqrt{p} \\  &=\  pqp.
\end{align*}
Thus,
if we repeat
the whole argument
with~$p$ replaced by~$\sqrt{p}$,
we see that $p\aandthen q = \sqrt{p}q\sqrt{p}$.

Let us now consider the general
case in which~$\mathscr{A}$
may not be $*$-isomorphic to~$\mathscr{B}(\mathscr{H})$
for some Hilbert space~$\mathscr{H}$,
but is instead
(without loss of generality)
 a von Neumann subalgebra
of~$\mathscr{B}(\mathscr{H})$ for some Hilbert space~$\mathscr{H}$
(see Theorem~\ref{thm:KGNS}).
Let~$q\in [0,1]_\mathscr{A}$.
Since the inclusion~$\varrho\colon \mathscr{A}\to\mathscr{B}(\mathscr{A})$
is a multiplicative process, we have
$\varrho(p \aandthen q) =  \varrho(p)\aandthen \varrho(q)
=\sqrt{\varrho(p)} \varrho(q)\sqrt{\varrho(p)}
= \varrho(\sqrt{p}q\sqrt{p})$.
Since~$\varrho$ is injective,
we conclude that~$ p \aandthen q = \sqrt{p} q \sqrt{p}$.
\end{proof}

\begin{proof}[Proof of Theorem~\ref{thm:main}]
By Proposition~\ref{prop:existence},
the sequential product~$*$
(given by $p*q=\sqrt{p}q\sqrt{p}$)
is an abstract sequential product,
and~$*$ is the only abstract sequential product
by Proposition~\ref{prop:uniqueness}
\end{proof}

\section*{Remarks}
\begin{rem}\label{rem:axioms}
Gudder and Lat\'emoli\`ere (G\&L) 
showed
in~\cite{gudder}
that the sequential product
on the effects of a Hilbert space~$\mathscr{H}$
is the only binary operation~$\aandthen$
that satisfies the following axioms.
For all~$a,b\in [0,1]_{\mathscr{B}(\mathscr{H})}$,
and every density operator~$\varrho$ on~$\mathscr{H}$,
\begin{enumerate}
    \item[GL1.] $\tr[(a \aandthen \varrho) b] = \tr[\varrho (a \aandthen b)]$;
    \item[GL2.] $a \aandthen 1 = 1 \aandthen a = a$;
    \item[GL3.] $a \aandthen (a \aandthen b) = (a \aandthen a) \aandthen b
                = a^2 \aandthen b$, and
    \item[GL4.] $a \mapsto a \aandthen b$ is strongly continuous.
\end{enumerate}
Let us compare their proof of uniqueness
with our proof of uniqueness of the abstract sequential product.
The broad strokes are similar:
in both proofs it is shown 
\begin{enumerate}
\item
first that
$p\aandthen q = \sqrt{p}u^* qu\sqrt{p}$
for appropriate~$u$;
\item
then that~$p^2 \aandthen q = pqp$
using Ax.4 and GL1 resp.,
\item
and finally $p \aandthen q = \sqrt{p}q\sqrt{p}$
  is obtained using GL3 and Ax.2 respectively.
  \end{enumerate}

However,
the short strokes are quite different.  
For instance,
while GL3 and Ax.2 clearly serve the same purpose in both proofs
(enabling the third step mentioned above),
the relation between GL1 and its analogue, Ax.4,
is less clear: Ax.4 only comes into play at the second step,
while GL1 is important in both the first and second steps.
Also,
the proof of G\&L has a branch in the first step
(case~iii on page~9 of~\cite{gudder}),
which has no companion in our proof.
\end{rem}

\begin{rem}\label{rem:cats}
The universal properties
of the compression~$c$ (from Proposition~\ref{prop:compression})
and of the corner~$\pi$ (from Proposition~\ref{prop:corner})
may be cast into the following chain of adjunctions.
\begin{equation*}
\label{diag:vn}
\vcenter{\xymatrix{
\int\square\ar[d]_{\dashv\;}^{\;\dashv}
   \ar@/_6ex/[d]^{\;\dashv}_{\begin{array}{c}\scriptstyle\mathrm{Compression} \\[-.5em]
                   \scriptstyle (p\in\Eff{\mathscr{A}}) \mapsto 
   \ceil{p^\perp}\mathscr{A}\ceil{p^\perp}\end{array}} 
   \ar@/^6ex/[d]_{\dashv\;}^{\begin{array}{c}\scriptstyle\mathrm{Corner} \\[-.5em]
                   \scriptstyle(p\in\Eff{\mathscr{A}}) \mapsto 
   \floor{p}\mathscr{A}\floor{p}\end{array}} \\
   \vN^\op\ar@/^3ex/[u]^(0.4){0\!}\ar@/_3ex/[u]_(0.4){\!1}
}}
\end{equation*}
Here,~$\vN$ is the category of von Neumann algebras
and
processes,
and
$\square$ is the functor~$\vN^\op \to \Cat{Poset}^\op$
given by~$\square (\mathscr{A}) = \Eff{\mathscr{A}}$
and~$\square(f)(p) = f(p^\perp)^\perp$,
and~$\int \square$ is its Grothendieck completion.
Such chains
appear in several other categories
and provide
a tool
to study the sequential product in other settings (see~\cite{telos}).
\end{rem}
\begin{rem}\label{rem:indep}
We have shown
that Ax.1, Ax.2, Ax.3, Ax.4
axiomatize the sequential product.
A natural question
is whether three of them would have sufficed.
We will show that
Ax.1, Ax.2, and Ax.4 cannot be dropped.
We do not know whether Ax.3 is redundant.

(Of course,
instead of being dropped,
the axioms may also be weakened. 
For example,
Ax.3 is used only
with~$f$ a representation, and Ax.4 is only used with~$e_1,e_2$ 
rank one projections.)
\begin{enumerate}
\item[Ax.1] The operation~$\aandthen$ given by
$p \aandthen q \equiv pqp$ satisfies Ax.2, Ax.3, and Ax.4,
but not Ax.1.
\item[Ax.2]
Observe that
if we pick
for every effect~$p$ 
on a von Neumann algebra~$\mathscr{A}$
a unitary~$u_p$ from~$\ceil{p}\mathscr{A}\ceil{p}$,
then we may form an operation~$\aandthen$
on all effects
by $p \aandthen q \equiv \sqrt{p} u_p^* q u_p \sqrt{p}$,
which satisfies Ax.1.

Further,
note that
if~$u_p^2 = u_{p^2}$ for all~$p$, then~$\aandthen$ satisfies~Ax.2;
and if~$f(u_p) = u_{f(p)}$ for any
unital~$*$-homomorphism~$f$,
    then~$\aandthen$ obeys Ax.3;
and if every~$u_p$ is self-adjoint, then~$\aandthen$ satisfies Ax.4.

Define~$u_p$
by~$u_p=g(p)$,
where $g\colon [0,1] \to \{-1,1\}$ 
is any
Borel function
with~$g(\nicefrac{2}{3})=1$
and $g(\nicefrac{4}{9})=-1$.
Then clearly~$\aandthen$ (defined by~$u_p$)
satisfies Ax.1, Ax.3, and Ax.4.
Also, $\aandthen$ does not satisfy Ax.2,
because 
for~$p = \smash{\left(\begin{smallmatrix}
1 & 0 \\ 0 & \nicefrac{2}{3}
\end{smallmatrix}\right)}$
in~$M_2$
we have~$u_p=
\smash{\left(\begin{smallmatrix}
1 & 0 \\
0 & 1
\end{smallmatrix} \right)}$,
while~$u_{p^2} = 
\smash{\left(\begin{smallmatrix}
1 & 0 \\
0 & -1
\end{smallmatrix} \right)}$,
and so~$(p\aandthen p)\aandthen q\neq p\aandthen(p\aandthen q)$,
where~$q=\frac{1}{\sqrt{2}}
\smash{\left(\begin{smallmatrix}
1 & 1 \\
1 & 1
\end{smallmatrix} \right)}$.

\item[Ax.4]
Pick a Borel function~$g\colon [0,1] \to S^1$
such that~$g(\nicefrac{1}{2})\neq 1$
and~$g(\lambda)^2 = g(\lambda^2)$
for all~$\lambda\in [0,1]$.
Then~$\aandthen$
given by $p\aandthen q= \sqrt{p}g(p)^*\, q\, g(p)\sqrt{p}$
obeys Ax.1, Ax.2, and Ax.3, but not Ax.4.
\end{enumerate}
\end{rem}
\begin{prob}
Do Ax.1, Ax.2 and Ax.4 imply Ax.3?
\end{prob}
\section*{Acknowledgements}
Sam Staton
suggested
we should look for a universal property for~$c$
from Proposition~\ref{prop:compression}.
Bart Jacobs noted the
chain of adjunctions
from Remark~\ref{rem:cats}.
Robert Furber pointed us towards
the spectral theorem
for the proof of Proposition~\ref{prop:compression}.
We thank them, and Kenta Cho
for their friendly help.

We have received funding from the
European Research Council under grant agreement \textnumero~320571.
\bibliography{main}{}
\bibliographystyle{alpha}

\appendix
\section{$C^*$-algebras}
\begin{trm}
        \begin{enumerate}
		\item
			A \textbf{$C^*$-algebra}~$\mathscr{A}$
			is a complete normed complex vector space
			endowed with a bilinear associative product
			and an antilinear map
			 $(-)^*\colon\mathscr{A}\to\mathscr{A}$
			 such that $a^{**}=a$, 
			 $(ab)^*=b^*a^*$,
			 $\|ab\|\leq \|a\|\|b\|$,
			 and$\|a^*a\|=\|a\|^2$
			 for all~$a,b\in\mathscr{A}$.

			 (The last equation is called the 
			 \emph{$C^*$-identity}.)
		\item
			An element $a$ of a $C^*$-algebra~$\mathscr{A}$
			is called 
			\begin{enumerate}
			\item
			\textbf{positive}
			if $a\equiv b^* b$ for some~$b\in \mathscr{A}$;
			\item
			\textbf{self-adjoint}
			if $a^* = a$;
			\item
			a \textbf{projection}
			if $a^*a=a$;
			\item
			\textbf{central}
			if~$ab=ba$ for all~$b\in\mathscr{A}$;
			\item
			a \textbf{unit}
			if~$ab=ba=b$ for all~$b\in\mathscr{A}$.
			\end{enumerate}
                        The set of positive elements of~$\mathscr{A}$
                        is denoted by~$\pos{\mathscr{A}}$,
                        and the set of self-adjoint elements of~$\mathscr{A}$
                         by~$\sa{\mathscr{A}}$.
		\item
			A $C^*$-algebra
			is partially ordered by as follows.
			For all~$a,b\in \mathscr{A}$,
			we have  $a\leq b$
			iff $b-a$ is positive.
		
		\item
			A $C^*$-algebra~$\mathscr{A}$ is 
			\begin{enumerate}
			\item
			\textbf{unital}
			if~$\mathscr{A}$ contains a unit,~$1$;
			\item
			\textbf{commutative}
			if~$ab=ba$ for all~$a,b\in\mathscr{A}$;
			\item 
			a \textbf{factor}
			if~$\mathscr{S}$ is unital
			and all its central elements are of the 
			form~$\lambda\cdot 1$ where~$\lambda\in \C$.
			\end{enumerate}
		\item
			Let~$\mathscr{A}$ and~$\mathscr{B}$
			be $C^*$-algebras.
			A linear map $f\colon \mathscr{A}\to\mathscr{B}$
			is called 
			\begin{enumerate}
			\item
			\textbf{bounded} if $\|f\|<\infty$,
			where
			\begin{equation*}
			\|f\|\ =\ \sup\{ \,\lambda\in [0,\infty)\colon\,
			\forall a\in\mathscr{A}[\ \|f(a)\|\leq \lambda \|a\|
			\ ]\,\}.
			\end{equation*}

			\item
			\textbf{contractive} if $\|f\|\leq 1$;
			\item
			a \textbf{$*$-homomorphism} if
			$f(ab)=f(a)f(b)$
			and $f(a^*)=f(a)^*$ for all~$a,b\in\mathscr{A}$;
			\item
			a \textbf{$*$-isomorphism}
			if~$f$ is a bijective $*$-homomorphism;
			\item
			\textbf{positive} if $f(a)\in\mathscr{B}_+$
			for all~$a\in\mathscr{A}_+$;
			\item
			\textbf{unital} if~$\mathscr{A}$ and~$\mathscr{B}$
			are unital, and~$f(1)=1$;
			\item
                        \textbf{normal}
                        if for every directed subset~$D$ of self-adjoint
                        elements of~$\mathscr{A}$:
                        if~$D$ has a supremum~$\bigvee D$ in~$\sa{\mathscr{A}}$,
                        then~$f(\bigvee D)$
                        is the supremum of~$\{f(d)\colon d\in D\}$
                        in~$\sa{\mathscr{B}}$;
			\item
			a \textbf{process}
			if~$f$ is normal, completely positive
				and contractive.
			\end{enumerate}
		\item
			Let~$\mathscr{A}$ be a unital $C^*$-algebra.
                         A \textbf{state} of~$\mathscr{A}$
                         is a positive unital 
                         linear map~$\varphi\colon \mathscr{A}\to
                         \C$.
		\item
			A \textbf{$C^*$-subalgebra}
			 of a $C^*$-algebra~$\mathscr{A}$
			is a norm closed linear subspace~$S$ of~$\mathscr{A}$
			such that $ab\in \mathscr{S}$ and $a^*\in \mathscr{S}$
			for all $a,b\in \mathscr{S}$.
			(Such a set~$\mathscr{S}$
			is itself a $C^*$-algebra
			in the obvious way.)
		\item
			For every positive element~$a$
			of a $C^*$-algebra~$\mathscr{A}$
			there is a unique positive $b\in\mathscr{A}$
			with~$a=b^2$ and $ba=ab$.
			We write $\sqrt{a}=b$.
	\end{enumerate}
\end{trm}
\begin{exa}
Let~$X$ be a compact Hausdorff space.
The commutative unital \emph{$C^*$-algebra
of continuous functions on~$X$}
is the set~$C(X)$ of continuous complex-valued functions on~$X$
 endowed with the supremum norm
and coordinatewise operations.
\end{exa}
\begin{thm}[Gel'fand--Neumark]
\label{thm:gelfand}
Every commutative unital $C^*$-algebra
is $*$-isomorphic to a $C^*$-algebra
of continuous functions on a compact Hausdorff space.
\end{thm}
\begin{proof}
Apply Theorem~2.1 of~\cite{conway1990}.
\end{proof}
\begin{exa}
\label{exa:BH}
Let~$\mathscr{H}$ be a Hilbert space.
The bounded operators
on~$\mathscr{H}$
form a unital $C^*$-algebra,~$\mathscr{B}(\mathscr{H})$,
in which the product is given by composition,
$(-)^*$ is the adjoint,
and the norm is the operator norm.
Moreover, $\mathscr{B}(\mathscr{H})$
is a factor
(of ``type I''),
and~$A\in\mathscr{B}(\mathscr{H})$
is positive iff $0\leq\left<x,Ax\right>$
for all~$x\in \mathscr{H}$.

A \emph{$C^*$-algebra of bounded operators on~$\mathscr{H}$}
is a $C^*$-subalgebra~$\mathscr{B}(\mathscr{H})$ 
of bounded operators on~$\mathscr{H}$
(but need not be a factor).
\end{exa}
\begin{thm}[Gel'fand--Neumark--Segal]
\label{thm:gns}
Every unital $C^*$-algebra
is $*$-isomorphic
to a $C^*$-algebra of bounded operators on a Hilbert space.
\end{thm}
\begin{proof}
Unfold Theorem 5.17 of~\cite{conway1990}.
\end{proof}

The norm determines the order:
\begin{lem}
\label{lem:order-via-norm}
Let~$\mathscr{A}$
be a unital $C^*$-algebra,
and  $a\in\mathscr{A}_\mathrm{sa}$.
Then
$a\geq 0$\  iff\  $\|\,\|a\|-a\,\|\,\leq\,\|a\|$.
\end{lem}
\begin{proof}
See VIII/Theorem~3.6 of~\cite{conway1990}.
\end{proof}

\begin{prop}
Let~$\mathscr{A}$ and~$\mathscr{B}$
be unital $C^*$-algebras,
and 
let~$f\colon\mathscr{A}\to\mathscr{B}$
be a unital $*$-homomorphism.

Then~$f$ is contractive,
and $f(\mathscr{A})$ is norm closed
and in fact a $C^*$-subalgebra of~$\mathscr{B}$.

Moreover,
if $f$ is injective,
then, for all~$a\in \mathscr{A}$,
we have 
 $\|f(a)\|=\|a\|$,
 and $f(a)\geq 0$ iff~$a\geq 0$.
\end{prop}
\begin{proof}
Use Theorem VIII/4.8 of~\cite{conway1990}
and Lem.~\ref{lem:order-via-norm}.
\end{proof}

If we apply the proposition above
to the inclusion of a $C^*$-subalgebra,
then we get the following
desirable result.
\begin{cor}
\label{cor:order-norm-permanence}
Let~$\mathscr{A}$ be a unital $C^*$-algebra,
and let $a$ be an element of a unital $C^*$-subalgebra~$\mathscr{B}$
of~$\mathscr{A}$.

Then~$\|a\|_{\mathscr{A}} = \|a\|_{\mathscr{B}}$,
and $a\in\mathscr{A}_+$ iff $a\in \mathscr{B}_+$.
\end{cor}

The order also determines the norm:
\begin{cor}
Let~$\mathscr{A}$ be a unital $C^*$-algebra.
Then 
\begin{equation}
\label{eq:norm-by-order}
\|a\| \ =\  \min\{\,\lambda\in [0,\infty)\colon\,
-\lambda\leq a\leq \lambda \,\}
\end{equation}
for any self-adjoint element~$a$ of~$\mathscr{A}$.
\end{cor}
\begin{proof}
Note that if~$\mathscr{A}=C(X)$ for some compact Hausdorff space,
then~\eqref{eq:norm-by-order} is evidently correct,
because the norm on~$C(X)$ is the supnorm.
Thus, \eqref{eq:norm-by-order}
is also correct if~$\mathscr{A}$ is commutative,
since in that case~$\mathscr{A}$
is $*$-isomorphic
to some~$C(X)$ by Theorem~\ref{thm:gelfand}.

In general, however, $\mathscr{A}$
need not be commutative,
but the $C^*$-subalgebra, $C^*(a)$,
generated by~$a$ is commutative.
Thus, since the order and the norm on~$C^*(a)$
agree with the order and norm on~$\mathscr{A}$
by Corollary~\ref{cor:order-norm-permanence},
\eqref{eq:norm-by-order}
holds on~$\mathscr{A}$
(because it holds on~$C^*(a)$).
\end{proof}

\begin{exa}
\label{exa:commutant-cstar}
Let~$\mathscr{A}$
be a $C^*$-algebra,
and let~$\mathscr{S}$ be a subset of~$\mathscr{A}$.
Then
$\mathscr{S}^\square = \{\,a\in\mathscr{A}\colon\, \forall
s\in \mathscr{S}\ [\, as=sa\,]\,\}$,
the \emph{commutant of~$\mathscr{S}$},
is a $C^*$-subalgebra of~$\mathscr{A}$
provided that $s^*\in\mathscr{S}$
for all~$s\in\mathscr{S}$.
\end{exa}
\begin{cor}
\label{cor:commutes-sqrt}
If an element, $a$,
of a $C^*$-algebra
commutes with~$b\geq 0$,
then~$a$ commutes with~$\sqrt{b}$.
\end{cor}
\begin{trm}
Let~$\mathscr{A}$
be a $C^*$-algebra
(of operators on a Hilbert space~$\mathscr{H}$)
and let~$N\in\N$.
By~$M_N(\mathscr{A})$
we denote the set of $N\times N$-matrices
over~$\mathscr{A}$
which is itself a $C^*$-algebra
(of operators on the 
Hilbert space~$\mathscr{H}^{\oplus N}$).

Let~$\mathscr{A}$ and~$\mathscr{B}$
be $C^*$-algebras.
Let $f\colon\mathscr{A}\to\mathscr{B}$
be a linear map.
We say that~$f$ is~\textbf{$N$-positive}
if for every positive $N\times N$-matrix
$(A_{ij})_{ij}$
over~$\mathscr{A}$ 
the $N\times N$-matrix
$(f(A_{ij}))_{ij}$ over~$\mathscr{B}$ 
is positive in~$M_N(\mathscr{B})$.
$f$ is \textbf{completely positive}
if~$f$ is $N$-positive for all~$N\in\N$.%
~\cite{stinespring1955}
\end{trm}
\begin{lem}
\label{lem:proj}
Let~$a$ be an element and $p$  a projection
in a unital $C^*$-algebra~$\mathscr{A}$.
If~$a^*a\leq p$, then $ap=a$.
\end{lem}
\begin{proof}
Follows from Lemma~\ref{lem:connected}.
\end{proof}
\begin{cor}
\label{cor:proj}
Let~$\mathscr{A}$
be a unital $C^*$-algebra.
For every projection~$p$ in~$\mathscr{A}$
and $a\in \mathscr{A}$ with~$0\leq a\leq p$,
we have $ap=pa=a$.
\end{cor}
\begin{cor}\label{lem:projsumunderone}
Let~$p,q$ be projections with $p+q\leq 1$
in a unital $C^*$-algebra.
Then  $pq = qp = 0$.
\end{cor}
\begin{lem}
\label{lem:projection-order}
For an element~$p$
of a unital $C^*$-algebra~$\mathscr{A}$,
the following are equivalent.
\begin{enumerate}
\item
\label{lem:projection-order-1}
$p$ is a projection.
\item
\label{lem:projection-order-2}
$a\leq p$ and~$a\leq 1-p$
entails~$a=0$ for all $a\in\mathscr{A}_+$.
\end{enumerate}
\end{lem}
\begin{proof}
\emph{(\ref{lem:projection-order-1}$\Longrightarrow$%
\ref{lem:projection-order-2})}\ 
Let~$a\in\mathscr{A}_+$ with $a\leq p$ and $a\leq 1-p$ be given.
Since~$ap=a$ and $a(1-p) =a$
by Corollary~\ref{cor:proj},
we get~$a=ap+a(1-p)=2a$,
and so~$a=0$.

\emph{(\ref{lem:projection-order-2}$\Longrightarrow$%
\ref{lem:projection-order-1})}\ 
We may assume that~$\mathscr{A}$
is commutative (by considering
the $C^*$-subalgebra generated by~$\{a\}$ instead),
and so~$\mathscr{A}\cong C(X)$
for some compact Hausdorff space
by Theorem~\ref{thm:gelfand}.

Then $a\in C(X)$ given by $a(x)=\min\{p(x),1-p(x)\}$
for all~$x\in X$
is positive and below both~$p$ and~$1-p$.
Thus~$a=0$ by assumption.
Then, for all~$x\in X$,
either~$p(x)=0$ or~$1-p(x)=0$.
Thus~$p$ takes only the values~$0$ and~$1$,
and is therefore easily seen to be a projection.
\end{proof}
\begin{cor}
\label{cor:isometry-preserves-projections}
Let~$f\colon \mathscr{A}\to\mathscr{B}$
be an invertible
positive unital linear map between unital $C^*$-algebras,
such that~$f^{-1}$ is positive.
Then~$f$ preserves projections.
\end{cor}
\section{Von Neumann Algebras}
\begin{trm}
	A \textbf{von Neumann algebra}
	is a unital $C^*$-algebra~$\mathscr{A}$
	such that: \emph{(I)} every bounded 
	directed set of self-adjoint
	elements of~$\mathscr{A}$
	has a supremum in~$\sa{\mathscr{A}}$,
	and \emph{(II)} for every positive $a\in\mathscr{A}$:
	if $\varphi(a)=0$ 
	for every \emph{normal} 
	state~$\varphi$ of~$\mathscr{A}$,
	then~$a=0$. \cite{kadison1956}

	A \textbf{von Neumann subalgebra}
	of a von Neumann algebra~$\mathscr{A}$
	is a $C^*$-subalgebra~$\mathscr{S}$ of~$\mathscr{A}$
	such that for every bounded directed set~$D$
	of~$\mathscr{S}_\mathrm{sa}$
	we have $\bigvee D \in \mathscr{S}$,
	where $\bigvee D$ is the supremum of~$D$ in~$\mathscr{A}_\mathrm{sa}$.
\end{trm}
\begin{trm}
	Let~$\mathscr{A}$
	be a $C^*$-algebra.
	Given a net~$(a_i)_i$
	in~$\mathscr{A}$
	and
	$b\in\mathscr{A}$,---
	\begin{enumerate}
		\item
			$(a_i)_i$
			converges \textbf{ultraweakly}
			to~$b$
			if for every normal state~$\varphi$
			of~$\mathscr{A}$,
			\begin{equation*}
				(\varphi(a_i))_i
				\quad\text{converges to}\quad 
				\varphi(b);
			\end{equation*}

		\item
			and --- provided~$\mathscr{A}$
			is a $C^*$-subalgebra
			of the space of bounded operators
			$\mathscr{B}(\mathscr{H})$
			on a Hilbert space~$\mathscr{H}$ ---
			$(a_i)_i$
			converges \textbf{weakly} to~$b$
			(with respect to~$\mathscr{H}$)
			if for all~$x\in \mathscr{H}$,
			\begin{equation*}
				(\left<a_ix,x\right>)_i
				\quad\text{converges to}\quad
				\left<bx,x\right>.
			\end{equation*}
	\end{enumerate}
\end{trm}
\begin{exa}
Let~$\mathscr{H}$
be a Hilbert space.
Then~$\mathscr{B}(\mathscr{H})$
is a von Neumann algebra.
\end{exa}
\begin{thm}[Kadison]
\label{thm:neumann-weakly-closed}
For a $C^*$-algebra of bounded operators 
on a Hilbert space,
the following are equivalent.
\begin{enumerate}
\item
$\mathscr{A}$ is a von Neumann subalgebra of~$\mathscr{B}(\mathscr{H})$;
\item
$\mathscr{A}$ is weakly closed in~$\mathscr{B}(\mathscr{H})$.
\end{enumerate}
\end{thm}
\begin{proof}
This follows from Lemma~1 of~\cite{kadison1956}.
\end{proof}
\begin{thm}[Kadison]
\label{thm:KGNS}
Any von Neumann algebra
is $*$-isomorphic
to a von Neumann subalgebra of~$\mathscr{B}(\mathscr{H})$
for some Hilbert space~$\mathscr{H}$.

Moreover, $\mathscr{H}$ can be chosen in such a way
that the ultraweak topology on~$\mathscr{A}$
coincides with weak topology on~$\mathscr{A}$
induced by~$\mathscr{B}(\mathscr{H})$
\end{thm}
\begin{proof}
That~$\mathscr{A}$ is $*$-isomorphic
to a von Neumann algebra of bounded
operators on some Hilbert space~$\mathscr{H}$
follows from Theorem~1 of~\cite{kadison1956}.
That the ultraweak topology on~$\mathscr{A}$
coincides with the weak topology on~$\mathscr{A}$
induced by~$\mathscr{H}$
follows
from the way the Hilbert space~$\mathscr{H}$
is constructed in the first paragraph of the proof of
Theorem~1\cite{kadison1956}
(if we take $(\omega_\alpha)_{\alpha \in \Gamma}$
to be the collection of all normal states):
for every normal state~$\omega$ of~$\mathscr{A}$
there is~$x \in \mathscr{H}$
with $\omega(a)=\left<x,ax\right>$
for all~$a\in \mathscr{A}$.
\end{proof}
\begin{exa}
Let~$X$ be a measure space.
Then the
$C^*$-algebra~$L^\infty(X)$
of bounded measurable complex-valued functions on~$X$
(in which two such functions are identified when they
are equal almost everywhere)
is a commutative von Neumann algebra
and the map $\varrho\colon L^\infty(X)\to\mathscr{B}(L^2(X))$
given by $\varrho(f)(g)=\int fgd\mu$
is an injective normal $*$-homomorphism,
where $L^2(X)$
is the Hilbert space of square integrable complex-valued functions on~$X$
(in which two such functions are identified when they are equal almost
everywhere).
\end{exa}
\begin{thm}[Spectral Theorem]
\label{thm:spectral}
For every self-adjoint bounded operator~$A$ on a Hilbert space~$\mathscr{H}$,
there is a measure space~$X$,
an element~$a$ of~$L^\infty(X)$,
and a unitary $U\colon L^2(X)\to \mathscr{H}$,
such that $U^* A U = \int a \cdot -\, d\mu$.
\end{thm}
\begin{proof}
See~\cite{halmos1963}.
\end{proof}
\begin{prop}
\label{prop:supremum-commutes}
Let~$D$ be a directed bounded set of self-adjoint
elements of a von Neumann algebra~$\mathscr{A}$.

Let~$b\in \mathscr{A}$.
If~$b$ commutes with all~$d\in D$,
then~$b$ commutes with~$\bigvee D$.
\end{prop}
\begin{proof}
We may assume 
(by Theorem~\ref{thm:KGNS})
without loss of generality that~$\mathscr{A}$
is a von Neumann subalgebra
of~$\mathscr{B}(\mathscr{H})$
for some Hilbert space~$\mathscr{H}$.
Since~$(d)_{d\in D}$ converges
strongly to~$\bigvee D$
(see Lemma~5.1.4 of~\cite{kadison1997})
we see that
$(bd)_{d\in D}$
converges weakly to~$b(\bigvee D)$.
Since $bd=db$ for all~$d\in D$,
and $(db)_{d\in D}$
converges weakly to $(\bigvee D)b$ by a similar reasoning,
we get $(\bigvee D)b = b (\bigvee D)$.
\end{proof}

\begin{prop}
\label{prop:ultraweak-convergence}
Let~$(a_i)_i$ be a net in a von Neumann algebra~$\mathscr{A}$
such that
\begin{enumerate}
\item $(a_i)_i$ is \emph{norm bounded},
that is $\sup_i \|a_i\| <\infty$, and
\item $(a_i)_i$ is \emph{ultraweakly Cauchy},
that is, $(\varphi(a_i))_i$ is Cauchy for every
normal state~$\varphi\colon \mathscr{A}\to \mathbb{C}$.
\end{enumerate}
Then~$(a_i)_i$ converges ultraweakly.
\end{prop}
\begin{proof}
By Theorem~\ref{thm:KGNS},
we may assume without loss of generality
that~$\mathscr{A}$
is a von Neumann algebra of bounded operators
on some Hilbert space~$\mathscr{H}$
such that the weak topology on~$\mathscr{A}$
induced by~$\mathscr{H}$
coincides with the ultraweak topology.

Let~$x\in\mathscr{H}$ be given.
Note that if~$\|x\|=1$,
then $\left<x,-\,x\right>\colon\mathscr{A}\to\C$
is a normal state,
and so~$(\left<x,a_ix\right>)_i$ is Cauchy.
It follows easily that $(\left<x,a_ix\right>)_i$
is Cauchy for all~$x\in \mathscr{H}$.

Let~$x,y\in\mathscr{H}$ be given.
Since for all $a\in\mathscr{A}$, 
\begin{equation*}
|\left<x,ay\right>|^2 \ \leq\ \left<x,ax\right>\, \left<y,ay\right>,
\end{equation*}
we see that $(\left<x,a_iy\right>)_i$ is Cauchy.

Since $(x,y)\mapsto \lim_i \left<x,a_iy\right>$
gives a bilinear map on~$\mathscr{H}$,
which is bounded because
$(a_i)_i$ is norm bounded,
there is, by Riesz's representation theorem,
 a bounded operator~$a$ on~$\mathscr{H}$
with~$\left<ax,y\right> = \lim_i \left<a_ix,y\right>$
for all $x,y\in\mathscr{H}$.

Note that~$(a_i)_i$ converges weakly to~$a$.
Thus~$a\in\mathscr{A}$,
because~$\mathscr{A}$ is weakly closed by 
Theorem~\ref{thm:neumann-weakly-closed}.
Further, $(a_i)_i$ converges ultraweakly to~$a$ as well,
because the weak and ultraweak topologies coincide on~$\mathscr{A}$
by choice of~$\mathscr{H}$.
\end{proof}
\begin{lem}\label{lem:comprcpn}
Let~$a$ be an element of a von Neumann algebra~$\mathscr{A}$.
Then the linear map~$c\colon \mathscr{A}\to\mathscr{A},\
b\mapsto a^*ba$
is normal and completely positive.
\end{lem}
\begin{proof}
        \emph{(Normality)}\ 
        follows from Lemma~1.7.4 of \cite{sakai1971}.

        \noindent
        \emph{(Complete positivity)}\ 
        follows from Theorem~1 of~\cite{stinespring1955},
	but let us give an elementary proof.

        Let~$N\in\N$ be given.
        Let~$B$ be a positive $N\times N$-matrix over~$\mathscr{A}$.
        We must show that~$(a^*B_{ij}a)_{ij}$
        is a positive $N\times N$-matrix
        over~$\mathscr{A}$.
        Since~$B$ is positive,
        there is a $N\times N$-matrix~$C$
        with~$B=C^*C$.
        Note that
        \begin{equation*}
                (a^* B_{ij} a)_{ij}\ = \ 
        A^* B A \ \equiv \ A^* C^* C A \ =\ (CA)^*CA \ \geq 0,
        \end{equation*}
        where $A=(a)_{ij}$
        is a diagonal $N\times N$-matrix.
        Thus~$c$ is completely positive.
\end{proof}
\begin{cor}
\label{cor:corner-von-neumann}
For every projection~$p$ of a von Neumann algebra~$\mathscr{A}$,
$p\mathscr{A}p$
is a von Neumann subalgebra of~$\mathscr{A}$.
\end{cor}
\begin{proof}
Surely, $p\mathscr{A}p$ is a $*$-subalgebra of~$\mathscr{A}$
with unit~$p$.
Since~$\|pap -pbp\|\leq \|p\|\|a-b\|\|p\|$
for all~$a,b\in\mathscr{A}$,
we see that~$p\mathscr{A} p$ is norm closed,
and~$p\mathscr{A}p$  is a $C^*$-subalgebra.

Let~$D$ be a bounded directed subset of~$(p\mathscr{A}p)_\mathrm{sa}$.
To prove that~$p\mathscr{A}p$ is a von Neumann subalgebra,
it suffices to show that the supremum~$\bigvee D$
of~$D$ in~$\mathscr{A}_\mathrm{sa}$ is in~$p\mathscr{A}p$.

Since~$a\mapsto pap$ is normal on~$\mathscr{A}$ by Lemma~\ref{lem:comprcpn},
and we have $d=pdp$ for all~$d\in D$,
we see that~$p(\bigvee D)p = \bigvee_{d\in D} pdp  = \bigvee D$,
and so~$\bigvee D \in \mathscr{A}$.
\end{proof}
\begin{prop}
        \label{prop:support-neumann}
        Let~$\mathscr{A}$
        be a von Neumann algebra.
        Let~$a\in \mathscr{A}$
        with $0\leq a \leq 1$
	be given.
        \begin{enumerate}
                \item
        \label{prop:support-neumann-1}
                        There is a smallest projection, $\ceil{a}$, above~$a$.
                \item
        \label{prop:support-neumann-2}
                        $\ceil{a}$ is the 
                        supremum of~$a \leq a^{\nicefrac{1}{2}}
                        \leq a^{\nicefrac{1}{4}} 
                        \leq a^{\nicefrac{1}{8}}
                        \leq\dotsb$.
                \item
        \label{prop:support-neumann-3}
			Then $ab=ba$ 
			implies $\ceil{a}b=b\ceil{a}$
			for all~$b\in\mathscr{B}$.
        \end{enumerate}
\end{prop}
\begin{proof}
Let~$p$ be the supremum of $a,\,a^{\nicefrac{1}{2}},\,
a^{\nicefrac{1}{4}},\,\dotsc$
in~$\mathscr{A}_\mathrm{sa}$.
Let~$q$ be a projection in~$\mathscr{A}$
with~$a\leq q$.
Then $aq=qa=a$ by Corollary~\ref{cor:proj},
and so $a^{\nicefrac{1}{2}}q=qa^{\nicefrac{1}{2}}$
by Corollary~\ref{cor:commutes-sqrt}.
Since~$a(1-q)=0$,
we have 
\begin{equation*}
\|\sqrt{a}(1-q)\|^2 \,=\,\|(1-q)a(1-q)\|=0
\end{equation*}
by the $C^*$-identity,
and so~$\sqrt{a}(1-q)=0$,
and thus~$\sqrt{a}q = \sqrt{a}$.
Then~$\sqrt{a}=\sqrt{a}q^2=q\sqrt{a}q\leq q$.
With a similar reasoning,
we get $a^{\nicefrac{1}{4}}\leq q$,
and $a^{\nicefrac{1}{8}}\leq q$, and so on.
It follows that~$p\leq q$,
by definition of~$p$.

Thus,
to show that~$p$ is the least projection above~$a$,
we only need to show that~$p$
is a projection.
Since~$0\leq p\leq 1$
(and thus $p^2\leq p$)
it suffices to show that~$p\leq p^2$.

First note that
any~$b\in\mathscr{A}$
that commutes with~$a$,
commutes with~$a^{\nicefrac{1}{2}}$,
and with~$a^{\nicefrac{1}{4}}$, etc.,
and thus~$b$ commutes with~$p$ 
by Proposition~\ref{prop:supremum-commutes}.

In particular,
since each~$a^{\nicefrac{1}{2^n}}$ commutes with~$a$,
we see that $a^{\nicefrac{1}{2^n}}$ commutes with~$p$.
Then, by Lemma~\ref{lem:comprcpn},
\begin{alignat*}{3}
p^2 \ &=\ \sqrt{p} p\sqrt{p} \\
&=\ \textstyle \bigvee_n \, \sqrt{p} \,a^{\nicefrac{1}{2^n}}\, \sqrt{p}\\
&=\ \textstyle \bigvee_n \, 
	a^{\nicefrac{1}{2^{n+1}}}\ p\ a^{\nicefrac{1}{2^{n+1}}}\\
&=\ \textstyle \bigvee_n\bigvee_m \, 
	a^{\nicefrac{1}{2^{n+1}}}\,a^{\nicefrac{1}{2^m}}
	\,a^{\nicefrac{1}{2^{n+1}}}.
\end{alignat*}
Thus $p^2 \geq a^{\nicefrac{1}{2^k}}$
for every~$k\in \N$,  and so~$p^2 \geq p$.

Hence~$p$ is a projection.
\end{proof}
\begin{prop}
\label{prop:support-ineq}
Let~$f\colon \mathscr{A}\to\mathscr{B}$
be a positive linear contraction between von Neumann algebras.
Let~$a\in \mathscr{A}$.

Then $f(\ceil{a}) \leq \ceil{f(a)}$,
and~$\ceil{f(\ceil{a})}=\ceil{f(a)}$.
\end{prop}
\begin{proof}
Since $\ceil{a}=\bigvee_n a^{\nicefrac{1}{2^n}}$
by Proposition~\ref{prop:support-neumann},
and~$f$ is normal,
we have
\begin{equation*}
\textstyle f(\ceil{a})\ = \ 
\bigvee_n f(a^{\nicefrac{1}{2^n}})
\ \stackrel{(*)}\leq \ \bigvee_n f(a)^{\nicefrac{1}{2^n}}
\ = \ \ceil{f(a)}. 
\end{equation*}
To justify Inequality~$(*)$
we claim that~$f(\sqrt{b})\leq \sqrt{f(b)}$ for all~$b\in \mathscr{B}_+$.
Since~$\sqrt{-}$ is order preserving\cite{pedersen1972},
it suffices to show that~$f(\sqrt{b})^2 \leq f(\smash{\sqrt{b}}^2)$,
and this has been done in Theorem~1 of~\cite{kadison1952}.

Let prove that $\ceil{f(\ceil{a})}=\ceil{f(a)}$.
On the one hand,
we have $\ceil{f(\ceil{a})}\geq \ceil{f(a)}$,
because $\ceil{a}\geq a$.
On the other hand,
since $\ceil{f(a)}$ is a projection,
and we have just shown that $f(\ceil{a})\leq \ceil{f(a)}$,
we get $\ceil{f(\ceil{a})}\leq \ceil{f(a)}$
by definition of~$\ceil{f(a)}$.
\end{proof}
%
%
%

\begin{thm}[Gardner]\label{thm:gardner}
    For a positive linear map~$f\colon \mathscr{A} \to \mathscr{B}$
    between unital~C$^*$-algebras, the following are equivalent.
    \begin{enumerate}
        \item[(ii)] $f(1)\cdot f(ab)\,=\,f(a)\cdot f(b)$
                   for all $a,b\in \mathscr{A}$.
        \item[(iii)$'$] $f$ is $2$-positive, and
	for all~$a,b\in \mathscr{A}_+$
		with
	$ab=0$ we have~$f(a)f(b)=0$.
    \end{enumerate}
\end{thm}
\begin{proof}
See Theorem~2 of~\cite{gardner}.
\end{proof}
\begin{prop}\label{prop:awmult}
    For a $2$-positive normal unital linear
     map~$f\colon \mathscr{A} \to \mathscr{B}$
        between von Neumann algebras
	the following are equivalent.
    \begin{enumerate}
     \item $f$ is a $*$-homomorphism.\label{prop:awmult-1}
     \item $f$ preserves projections.
     \label{prop:awmult-2}
    \item $\ceil{f(a)} = f(\ceil{a})$ for every~$a\in[0,1]_\mathscr{A}$.
    \label{prop:awmult-3}
    \end{enumerate}
\end{prop}
\begin{proof}
    \emph{(\ref{prop:awmult-1} $\Longrightarrow$ \ref{prop:awmult-2})}\ 
    Easy.

    \emph{(\ref{prop:awmult-2} $\Longrightarrow$ \ref{prop:awmult-3})}\ 
     Let~$a\in [0,1]_\mathscr{A}$ be given.
     By Proposition~\ref{prop:support-ineq}
     we have $\ceil{f(a)}= \ceil{f(\ceil{a})} = f(\ceil{a})$,
     where the latter equality follows from
     the fact that~$f(\ceil{a})$ is a projection.

    \emph{(\ref{prop:awmult-3} $\Longrightarrow$ \ref{prop:awmult-1})}\ 
    Let $a,b\in\mathscr{A}_+$ with $ab=0$ be given.
    To prove that~$f$ is multiplicative,
    it suffices to show that $f(a)f(b)=0$
    by Theorem~\ref{thm:gardner}
    (since~$f(1)=1$).

    If either~$a$ or~$b$ is zero, we are done,
    so we may assume that~$a\neq 0$ and~$b\neq 0$.
    We may also assume that~$a,b \leq 1$
    (by replacing them by 
    $\nicefrac{a}{\|a\|}$ and~$\nicefrac{b}{\|b\|}$
    if necessary).

    It suffices to show that $\ceil{f(a)}\,\ceil{f(b)}=0$,
    because then $f(a)f(b)= f(a)\ceil{f(a)}\,\ceil{f(b)}f(b)=0$,
    where we used that~$f(a)=f(a)\ceil{f(a)}$ (see 
Proposition~\ref{prop:support-neumann}).

    Note that~$a$ and~$b$ commute, because~$ba = b^*a^* = (ab)^* = 0 = ab$.
    Then $\sqrt{a}$ and~$\sqrt{b}$ commute as well,
    and so~$\sqrt{a}b\sqrt{a}=ab=0$.
    Then~$0\leq \sqrt{a} \ceil{b}\sqrt{a} \leq \ceil{\sqrt{a}b\sqrt{a}}=0$,
    and so~$a\ceil{b}=0$.
    By repeating this argument,
    we see that~$\ceil{a}\ceil{b}=0$.

    It follows that~$\ceil{a} + \ceil{b} $ is a projection, and
     \begin{equation*}
      \ceil{f(a)} \,+\, \ceil{f(b)}\ =\ 
    f(\ceil{a}) \,+\, f(\ceil{b}) \ \leq\  f(1) \ =\  1.
    \end{equation*}
    Thus Corollary~\ref{lem:projsumunderone}
        implies that~$\ceil{f(a)}\,\ceil{f(b)} = 0$.
\end{proof}
\begin{cor}\label{cor:invprocmult}
Let~$f$
be an invertible process between von Neumann
algebras such that~$f^{-1}$ is a process as well.
Then~$f$ is a $*$-isomorphism.
\end{cor}
\begin{proof}
Since~$f(1)\leq 1 = f(f^{-1}(1))$
we have $1\leq f^{-1}(1)\leq 1$,
and so~$f^{-1}(1)=1$.
Thus both~$f$ and~$f^{-1}$ are unital.
Then
$f$ preserves projections
by Corollary~\ref{cor:isometry-preserves-projections},
and is thus a $*$-homomorphism by Proposition~\ref{prop:awmult}.

Hence~$f$ is a $*$-isomorphism.
\end{proof}
\section{Ultraweak limits of maps}
\begin{lem}
        \label{lem:normal-uw}
        For a positive linear map~$f\colon \mathscr{A}\to\mathscr{B}$
	between von Neumann algebras
        the following are equivalent.
        \begin{enumerate}
                \item 
                        \label{lem:normal-uw-normal}
                        $f$ is normal.
                \item
                        \label{lem:normal-uw-uw}
                        $f$ is ultraweakly continuous.
                \item
                        \label{lem:normal-uw-uw-eff}
                        The restriction of~$f$
                        to a map~$\Eff{\mathscr{A}}\to\mathscr{B}$
                        is ultraweakly continuous.
        \end{enumerate}
\end{lem}
\begin{proof}
        \emph{(\ref{lem:normal-uw-normal}
                $\Longrightarrow$\ref{lem:normal-uw-uw})}\ 
        Let~$\varphi\colon \mathscr{B}\to \C$
        be a normal state.
        To prove that~$f$ is ultraweakly continuous
        we must show that
        $\varphi\circ f\colon \mathscr{A}\to\C$ 
        is continuous with respect to the ultraweak
        topology on~$\mathscr{A}$ and the standard topology on~$\C$.
        It suffices to show that~$\varphi\circ f$
        is normal,
        which indeed it is, 
        as both~$\varphi$ and~$f$ are normal.

        \emph{(\ref{lem:normal-uw-uw}
                $\Longrightarrow$\ref{lem:normal-uw-uw-eff})}\ 
        is trivial.

        \emph{(\ref{lem:normal-uw-uw-eff}
                $\Longrightarrow$\ref{lem:normal-uw-normal})}\ 
        Let~$D$
        be a bounded directed set of self-adjoint elements of~$\mathscr{A}$
        with supremum,~$\bigvee D$.
        Then as~$f$ is positive, $\{\,f(d)\colon\, d\in D \,\}$
        is directed 
        and bounded by~$f(\,\bigvee D\,)$,
        and thus has a supremum, $\bigvee_{d\in D} f(d)$.
        To show that~$f$ is normal,
        we must prove that~$f(\,\bigvee D\,) = \bigvee_{d\in D} f(d)$.
        Since~$f$ is linear,
        we may assume without loss of generality
        that~$D\subseteq\Eff{\mathscr{A}}$.
        Let~$\varphi\colon\mathscr{B}\to\C$
        be a normal state.
        It suffices to show that
        \begin{equation}
                \label{eq:lem:normal-uw-1}
                \textstyle
                \varphi(f(\bigvee D)) \ = \ 
                \varphi(\bigvee_{d\in D} f(d)).
        \end{equation}
        Note that~$D$ (as net) converges ultraweakly to~$\bigvee D$
        in~$\mathscr{A}$,
        and thus in~$\Eff{\mathscr{A}}$ as well.
        Since the restriction of~$f$ to~$\Eff{\mathscr{A}}$
        is ultraweakly continuous,
        the net $(f(d))_{d\in D}$
        converges ultraweakly to~$f(\,\bigvee D\,)$ in~$\mathscr{B}$.
        So~$(\varphi(f(d)))_{d\in D}$
        converges to~$\varphi(f(\bigvee D))$.
        Since~$(\varphi(f(d)))_{d\in D}$
        is directed,
        $\varphi(f(\bigvee D))$
        is in fact its supremum.
        Finally, since~$\varphi$ is normal,
        $ \varphi(\,\bigvee_{d\in D}f(d)\,)
        = \bigvee_{d\in D} \varphi(f(d))
        = \varphi(f(\,\bigvee D\,))$.
        We have proven Statement~\eqref{eq:lem:normal-uw-1},
        so~$f$ is normal.
\end{proof}

\begin{cor}
        \label{cor:uw-lim-normal}
        Let~$f\colon\mathscr{A}\to\mathscr{B}$
        be a positive linear map
	between von Neumann algebras.
        Let~$(f_\alpha)_{\alpha\in D}$
        be a net of normal positive linear maps
        from~$\mathscr{A}$ to~$\mathscr{B}$
        which converges uniformly on~$\Eff{\mathscr{A}}$ ultraweakly
        to~$f$.

        Then~$f$ is normal.
\end{cor}
\begin{proof}
        The uniform limit of continuous functions is continuous.
        In particular,
        since the~$f_\alpha$
        (being normal and hence ultraweakly continuous)
        converge uniformly on~$\Eff{\mathscr{A}}$
        to~$f$,
        we see that 
        the restriction of~$f$ to~$\Eff{\mathscr{A}}$
        is ultraweakly continuous,
        and thus~$f$ is normal by Lemma~\ref{lem:normal-uw}.
\end{proof}

\begin{lem}
        \label{lem:weak-lim-cp}
        Let~$\mathscr{B}$
        be a $C^*$-algebra
        of operators on a Hilbert space~$\mathscr{H}$.
        Let~$\mathscr{A}$
        be a $C^*$-algebra.
        Let~$(f_\alpha)_{\alpha\in D}$
        be a net of completely positive linear 
        maps from~$\mathscr{A}$ to~$\mathscr{B}$
        which converges pointwise weakly
        to a linear map~$f\colon \mathscr{A}\to\mathscr{B}$.
        Then~$f$ is completely positive.
\end{lem}
\begin{proof}
        Let~$A$
        be a positive $N\times N$-matrix
        over~$\mathscr{A}$
	for some~$N\in\N$.
        We must show that
        $(f(A_{ij}))_{ij}$
        is a positive $N\times N$-matrix over~$\mathscr{B}$.
        Note that the~$N\times N$-matrices over~$\mathscr{B}$
        can be considered a $C^*$-subalgebra
        of operators on~$\mathscr{H}^{\oplus N}$.
        To prove that~$(f(A_{ij}))_{ij}$ is positive,
        we will show that~$(f_\alpha(A_{ij}))_{ij}$
        converges to~$(f(A_{ij}))_{ij}$ weakly
	with respect to~$\mathscr{H}^{\oplus N}$.
        (This is sufficient, because
	the weak limit of positive operators is positive,
	and
        each~$(f_\alpha(A_{ij}))_{ij}$ is positive.)

        Let~$x,y\in \mathscr{H}^{\oplus N}$
        be given.
        To show that $(f_\alpha(A_{ij}))_{ij}$
        converges to~$(f(A_{ij}))_{ij}$ in the weak operator topology
        we must show that
	\begin{equation}
                \label{eq:weak-limit-cp-1} 
        \begin{alignedat}{3}
		&
                \left<\  (f(A_{ij})-f_\alpha(A_{ij}) )_{ij}\ x,\ 
                   y\ \right> \\
            & \qquad
                   \ \equiv\ 
                   {\textstyle \sum_{i,j} }\left<\ 
                   (f(A_{ij})-f_\alpha(A_{ij}))\ x_j,\ 
                   y_i\ \right>
        \end{alignedat}
	\end{equation}
        converges to~$0$ as~$\alpha\rightarrow \infty$.
        Let~$i,j\in\{1,\dotsc,N\}$ be given. Since $f_\alpha$
        converges pointwise in the weak operator topology to~$f$,
        $\left<(f_\alpha(A_{ij})-f(A_{ij}))x_j,y_i\right>$ converges in~$\C$
        to~$0$.
        Thus the right-hand side of Equality~\eqref{eq:weak-limit-cp-1},
        being a finite sum of such terms, 
        converges to~$0$ as~$\alpha\to\infty$.
        Thus~$f$ is completely positive.
\end{proof}
\begin{cor}
        \label{cor:uw-lim-cp}
        Let~$\mathscr{A}$ and~$\mathscr{B}$
	be von Neumann algebras.
        Let~$(f_\alpha)_{\alpha\in D}$
        be a net of completely positive linear maps
        from~$\mathscr{A}$ to~$\mathscr{B}$
        which converges pointwise ultraweakly
        to a linear map $f\colon\mathscr{A}\to\mathscr{B}$.
        Then~$f$ is completely positive.
\end{cor}

\section{Cauchy--Schwarz for~$2$-Positive Maps}
The classical form of the Cauchy--Schwarz inequality
is 
that for any vectors~$x$ and~$y$ in 
a complex vector space~$\mathscr{X}$
with semi-inner product~$\left<-,-\right>$
we have
\begin{equation*}
        \left|\left<x,y\right>\right|^2\ 
        \leq\ \left<x,x\right>\,\left<y,y\right>.
\end{equation*}
Since
any positive functional~$\varphi$
on a $C^*$-algebra~$\mathscr{A}$
gives a semi-inner product
on~$\mathscr{A}$  by~$\left<a,b\right>=\varphi(a^*b)$,
\begin{equation}
        \left| \varphi(a^*b)\right|^2 
                \ \leq\ \varphi(a^*a)\, \varphi(b^*b).\label{eq:kadineq}
\end{equation}
This is known as \emph{Kadison's inequality}.
We need the following generalization.
Given a $2$-positive linear map~$\varphi\colon \mathscr{A}\to\mathscr{B}$
we have,
for all~$a,b\in \mathscr{A}$,
\begin{equation}
        \label{eq:cs-teaser}
        \|\varphi(a^*b)\|^2\ \leq\ 
        \|\varphi(a^*a)\|\, \|\varphi(b^*b)\|.
\end{equation}
Since it
is an exercise in~\cite{paulsen2002}
and seems not to be mentioned elsewhere
we have included a proof
of Inequality~\eqref{eq:cs-teaser}
in this subsection (see Theorem~\ref{thm:cs}).

Recall that a linear map~$\varphi\colon \mathscr{A}\to\mathscr{B}$
is $2$-positive
whenever $\smash{\bigl[\begin{smallmatrix}\varphi(a) & \varphi(b) \\
\varphi(c) & \varphi(d)
\end{smallmatrix}\bigr]}$
is positive
for every
positive matrix
 $\left[\begin{smallmatrix}a & b \\
c & d
\end{smallmatrix}\right]$
with~$a,b,c,d\in\mathscr{A}$.
The trick behind the proof of Inequality~\eqref{eq:cs-teaser}
is to analyze
which $2\times 2$ matrices of operators
on a Hilbert space are positive
(see Lemma~\ref{lem:2-times-2-pos}).
Let us first recall the situation
for $2\times 2$-matrices over~$\C$.
\begin{lem}
        \label{lem:2-times-2-pos-C}
        Let~$T\equiv 
        \left[\begin{smallmatrix}
                        p & a \\
                        a^* & q 
        \end{smallmatrix}\right]$
        be a self-adjoint $2\times 2$ matrix over~$\C$.
        The following are equivalent.
        \begin{enumerate}
                \item $T$ is positive;
                \item $T$ has positive eigenvalues;
                \item $T$ has positive determinant and positive trace;
                \item 
                        $p,q\geq 0$ and
                        $|a|^2\leq pq$.
        \end{enumerate}
\end{lem}
\begin{proof}
We leave this to the reader.
\end{proof}
\begin{lem}
        \label{lem:2-times-2-pos}
        Let~$T\equiv
        \left[
        \begin{smallmatrix}
                P & A \\
                A^* & Q
        \end{smallmatrix}\right]$
         be a self-adjoint $2\times 2$ matrix
        of bounded operators on a Hilbert space~$\mathscr{H}$.
        The following are equivalent.
        \begin{enumerate}
                \item 
                        \label{lem:2-times-2-pos-1}
                        $T$ is positive.
                \item 
                        \label{lem:2-times-2-pos-2}
                        $P,Q\geq 0$, and for all~$x,y\in\mathscr{H}$,
                \begin{equation}
                        \label{eq:pos-2-times-2}
                        \left|\left< Ay,x\right>\right|^2\ \leq\ 
                        \left<Px,x\right>\ \left<Qy,y\right>.
                \end{equation}
                \newcounter{enumTemp}
                \setcounter{enumTemp}{\theenumi}
        \end{enumerate}
        Moreover,
        if~$T$ is positive, then:
        \begin{enumerate}
                        \setcounter{enumi}{\theenumTemp}
                \item
                        \label{lem:2-times-2-pos-3}
                        $A^*A \leq \|P\|\, Q$
                \item
                        \label{lem:2-times-2-pos-4}
                        $AA^* \leq \|Q\|\,P$
                \item $\|A\|^2 \leq \|P\|\, \|Q\|$
                        \label{lem:2-times-2-pos-5}
        \end{enumerate}
\end{lem}
\begin{proof}
        \emph{(\ref{lem:2-times-2-pos-1}$\ 
        \ \Longrightarrow\ $\ref{lem:2-times-2-pos-2})}\ 
        Let~$x,y\in\mathscr{H}$
        be given.
        Let us consider 
        $T' := \bigl[\smash{\begin{smallmatrix}
                        \left<Px,x\right>&
                        \left<Ay,x\right>\\
                        \left<A^*x,y\right>&
                        \left<Qy,y\right>
        \end{smallmatrix}}\bigr]$.
        Since~$T$ is self-adjoint,
        $T'$ is self-adjoint.
        Further,
        given~$\lambda,\mu\in\C$
        we have
        \begin{equation*}
                \label{eq:pos-2-times-2-1}
                \Bigl<
                \begin{bmatrix}
                        P & A \\
                        A^*& Q
                \end{bmatrix}\!\!
                \begin{bmatrix}
                \lambda x \\ \mu y
                \end{bmatrix},
                \begin{bmatrix}
                        \lambda x \\ \mu y 
                \end{bmatrix} 
                \Bigr> 
                    =
                \Bigl<
                \begin{bmatrix}
                        \left<Px,x\right>&
                        \left<Ay,x\right>\\
                        \left<A^*x,y\right>&
                        \left<Qy,y\right>
                \end{bmatrix}\!\!
                \begin{bmatrix}
                        \lambda \\ \mu
                \end{bmatrix},
                \begin{bmatrix}
                \lambda \\ \mu
                \end{bmatrix}
                \Bigr>.
        \end{equation*}
        From this we see that
        as $T$ is positive, $T'$ is positive.
        Then by Lemma~\ref{lem:2-times-2-pos-C}
        we get~$\left<Px,x\right> \geq 0$,
        $\left<Qy,y\right> \geq 0$,
        and
        $\left|\left<Ay,x\right>\right|^2 \leq
        \left<Px,x\right>\,\left<Qy,y\right>$.
        Hence~$P$ and~$Q$ are positive,
        and Inequality~\eqref{eq:pos-2-times-2} holds.
        \vspace{.3em}

        \noindent
        \emph{(\ref{lem:2-times-2-pos-2}\ $\Longrightarrow$
                \ref{lem:2-times-2-pos-1})}\ 
        We must show that~$T$ is positive.
        Note that~$T$ is self-adjoint
        since both~$P$ and~$Q$ are self-adjoint.
        Given~$x,y\in\mathscr{H}$
        we have
        \begin{equation*}
                \left< 
                \begin{bmatrix}
                        P & A \\
                        A^*& Q
                \end{bmatrix}
                \begin{bmatrix}
                      x \\ y
                \end{bmatrix},
                \begin{bmatrix}
                        x \\ y 
                \end{bmatrix} 
                \right> 
                = 
                \left<
                \begin{bmatrix}
                        \left<Px,x\right>&
                        \left<Ay,x\right>\\
                        \left<A^*x,y\right>&
                        \left<Qy,y\right>
                \end{bmatrix}
                \begin{bmatrix}
                        1 \\ 1
                \end{bmatrix}
                ,
                \begin{bmatrix}
                1 \\ 1
                \end{bmatrix}
                \right>.
        \end{equation*}
        So to show that~$T$ is positive,
        it suffices to show that
        $T' := \bigl[\smash{\begin{smallmatrix}
                        \left<Px,x\right>&
                        \left<Ay,x\right>\\
                        \left<A^*x,y\right>&
                        \left<Qy,y\right>
        \end{smallmatrix}}\bigr]$
        is positive.
        By Lemma~\ref{lem:2-times-2-pos-C}
        we must show that~$\left<Px,x\right> \geq 0$,
        $\left<Qy,y\right>\geq 0$,
        and $\left|\left<Ay,x\right>\right|^2 \leq \left<Px,x\right>
        \left<Qy,y\right>$.
        The latter statement is Inequality~\eqref{eq:pos-2-times-2}
        and holds by assumption.
        The other two statements follow from~$P\geq 0$ and~$Q\geq 0$.
        \vspace{.3em}

        \noindent
        \emph{(\ref{lem:2-times-2-pos-3})}\ 
        Assume that~$T$ is positive. 
        Let~$y\in\mathscr{H}$ be given. We must show that
        \begin{equation}
                \label{eq:2-times-2-pos-3}
                \left<A^*Ay,y\right>\ \leq\ \|P\|\,\left<Qy,y\right>.
        \end{equation}
        Note that $0\leq \left<A^*Ay,y\right> = \left<Ay,Ay\right>
        = \left|\left<Ay,Ay\right>\right|$.
	So
        \begin{align*}
                \left|\left<Ay,Ay\right>\right|^2\ &\leq\ \left<PAy,Ay\right>\,
                        \left<Qy,y\right>\qquad \\
                        &\qquad\qquad \text{by Ineq.~\eqref{eq:pos-2-times-2} with~$x=Ay$}\\
                        &\leq \ \|P\|\left<Ay,Ay\right>\,
                        \left<Qy,y\right>\qquad \\
                        & \qquad\qquad \text{since $P\leq \|P\|$ and $0\leq Q$}.
        \end{align*}
        So either~$\left<A^*Ay,y\right>=0$ --- in 
        which case Inequality~\eqref{eq:2-times-2-pos-3}
        holds trivially --- or~$\left<A^*Ay,y\right>\neq 0$
        in which case we get
        \begin{equation*}
        \left<A^*Ay,y\right>=\left<Ay,Ay\right>\ 
        \leq\ \|P\|\left<Qy,y\right>.
        \end{equation*}
        Thus~$A^*A \leq \|P\|\,Q$.
        \vspace{.3em}

        \noindent
        \emph{(\ref{lem:2-times-2-pos-4})}\ 
        follows by a similar reasoning as in~\ref{lem:2-times-2-pos-3}.
        \vspace{.3em}

        \noindent
        \emph{(\ref{lem:2-times-2-pos-5})}\ 
	We have
        $\|A\|^2=\|A^*A\|\leq \|\,\|P\|Q\,\|=\|P\|\,\|Q\|$
        since~$A^*A\leq \|P\|Q$ by~\ref{lem:2-times-2-pos-3}.
\end{proof}
\begin{thm}[Cauchy--Schwarz for~$2$-positive maps]$\,$ 
        \label{thm:cs}
        Let~$f\colon \mathscr{A}\to\mathscr{B}$
        be a $2$-positive map
        between $C^*$-algebras.
        Then we have, for all~$a,b\in \mathscr{A}$:
        \begin{enumerate}
                \item 
                        $f(b^*a)\,f(a^*b)\ \leq\ \|f(a^*a)\|\,f(b^*b)$
                \item
                        $f(a^*b)\,f(b^*a)\ \leq\ \|f(b^*b)\|\,f(a^*a)$
                \item
                        $\|f(a^*b)\|^2 \ \leq\ \|f(a^*a)\|\,\|f(b^*b)\|$
        \end{enumerate}
\end{thm}
\begin{proof}
        We may assume that~$\mathscr{B}$
        is a $C^*$-subalgebra of the space of 
        bounded linear operators~$\mathscr{B}(\mathscr{H})$
        on Hilbert space~$\mathscr{H}$.
        Since~$
        \bigl[\begin{smallmatrix}
                        a^*a & a^*b \\
        b^*a & b^*b \end{smallmatrix}\bigr]
                        \equiv 
        \bigl[\begin{smallmatrix}
                        a^* & 0 \\
        b^* & 0 \end{smallmatrix}\bigr]
                        \,
        \bigl[\begin{smallmatrix}
                        a & b \\
        0 & 0 \end{smallmatrix}\bigr]$
                        is positive 
        and~$f$ is $2$-positive
        we get that
        $T:=\bigl[\smash{\begin{smallmatrix}
                        f(a^*a) & f(a^*b) \\
        f(b^*a) & f(b^*b) \end{smallmatrix}}\bigr]$
                        is positive in~$M_2(\mathscr{B})$,
                        and thus~$T$ is positive 
                        in~$M_2(\mathscr{B(\mathscr{H})})$.

                        Now apply Lemma~\ref{lem:2-times-2-pos}.
\end{proof}

\end{document}